\documentclass[12pt]{amsart}
\usepackage{amssymb, colordvi}
\usepackage{multirow}
\usepackage{graphicx}
\numberwithin{equation}{section}
\newtheorem{thm}{Theorem}
\numberwithin{thm}{section}
\newtheorem{prop}[thm]{Proposition}
\newtheorem{lemma}[thm]{Lemma}
\newtheorem{cor}[thm]{Corollary}

\newtheorem{example}[thm]{Example}
\newtheorem{remark}[thm]{Remark}

\newenvironment{ex}{\begin{example}\rm}{\end{example}}

\newcounter{FNC}[page]
\def\fauxfootnote#1{{\addtocounter{FNC}{2}$^\fnsymbol{FNC}$%
     \let\thefootnote\relax\footnotetext{$^\fnsymbol{FNC}$#1}}}

\newcommand{\calS}{\mathcal{S}}

\newcommand{\calV}{\mathcal{V}}
\newcommand{\calW}{\mathcal{W}}

\usepackage{xcolor}

\newcommand{\ind}{\mathbf{1}}
\newcommand{\bW}{\mathbf{W}}

\newcommand{\bV}{\mathbf{V}}

\newcommand{\R}{\mathbb{R}}
\newcommand{\Z}{\mathbb{Z}}
\newcommand{\KK}{\mathbb{K}}

\title[Fewnomial bounds for tropical polynomial systems]{Irrational mixed decomposition and sharp fewnomial bounds for tropical polynomial systems}
\author{Fr\'ed\'eric Bihan}
\address{Laboratoire de Math\'ematiques\\
         Universit\'e de Savoie\\
         73376 Le Bourget-du-Lac Cedex\\
         France}

\email{Frederic.Bihan@univ-savoie.fr}
\urladdr{http://www.lama.univ-savoie.fr/~bihan/}

\begin{document}
\maketitle
\begin{abstract}

Given convex polytopes $P_1,\ldots,P_r \subset \R^n$ and finite subsets $\calW_I$ of the Minkowsky sums $P_I=\sum_{i \in I} P_i$, we consider the quantity
$N(\bW)=\sum_{I \subset {\bf [}r {\bf ]} } {(-1)}^{r-|I|} \big| \calW_I \big|$. If $\calW_I=\Z^n \cap P_I$ and $P_1,\ldots,P_n$ are lattice polytopes in $\R^n$, then $N(\bW)$ is the classical {\it mixed volume} of $P_1,\ldots,P_n$ giving the number of complex solutions of a general complex polynomial system with Newton polytopes $P_1,\ldots,P_n$.
We develop a technique that we call {\it irrational mixed decomposition} which allows us to estimate
$N(\bW)$ under some assumptions on the family $\bW=(\calW_I)$. In particular, we are able to show the nonnegativity of $N(\bW)$ in some important cases. 
A special attention is paid to the family $\bW=(\calW_I)$ defined by $\calW_I=\sum_{i \in I} \calW_i$,
where $\calW_1,\ldots,\calW_r$ are finite subsets of $P_1,\ldots,P_r$. The associated quantity
$N(\bW)$ is called {\it discrete mixed volume} of $\calW_1,\ldots,\calW_r$.
Using our irrational mixed decomposition technique, we show that for $r=n$ the discrete mixed volume is an upper bound for the number of nondegenerate solutions of a tropical polynomial system with supports $\calW_1,\ldots,\calW_n \subset \R^n$.
We also prove that the discrete mixed volume associated with  $\calW_1,\ldots,\calW_r$ is bounded from above by the Kouchnirenko number $\prod_{i=1}^r (|\calW_i|-1)$. For $r=n$ this number was proposed as a bound for the number of nondegenerate positive solutions of any real polynomial system with supports $\calW_1,\ldots,\calW_n \subset \R^n$.
This conjecture was disproved, but our result shows that the Kouchnirenko number is a sharp bound for the number of nondegenerate positive solutions
of real polynomial systems constructed by means of the combinatorial patchworking.
\end{abstract}
\section{Introduction}


It follows from Descartes' rule of signs that a real univariate polynomial with $N$ monomials has at most $N-1$ positive roots. This {\it  Descartes bound} is sharp and a major problem in real geometry consists in generalizing such a sharp bound to the multivariable case. 
The {\it support} of a Laurent polynomial $f(x)=\sum c_w x^w \in \R[x_1^{\pm 1},\ldots,x_n^{\pm 1}]$ is the set of exponent vectors $w \in \Z^n$ with non-zero coefficient $c_w$.
Consider a system $f_1=\cdots=f_n=0$ of Laurent polynomial equations with real coefficients in $n$ variables. A solution of the system is {\it  positive} if all its coordinates are positive and {\it  nondegenerate} if the jacobian determinant of $f_1,\ldots,f_n$  does not vanish at this solution. Denote by $\calW_i$ the support of $f_i$ and by $|\calW_i|$ its number of elements. Several conjectures have been made towards a generalized Descartes bound.
In the mid 70's, A. Kouchnirenko proposed $$\prod_{i=1}^n(|\calW_i|-1)$$
as a bound for the number of nondegenerate positive solutions of any real poynomial system with supports $\calW_1,\ldots,\calW_n$. Thanks to Descartes bound, the previous Kouchnirenko bound is the sharp bound for systems where each variable appears in one and only one equation. Kouchnirenko conjecture has been disproved a bit later by a russian student K. Sevostyanov. The counter-example was lost but a counter-example of the same kind was found in 2002 by B. Hass ~\cite{Ha}. Both counter-examples are systems of two trinomial equations in two variables with $5>(3-1)(3-1)$ nondegenererate positive solutions. 
T. Li, J.-M. Rojas and X. Wang showed that $5$ is in fact the sharp bound for this kind of polynomial system~\cite{LRW03}.
%
In 1980, A. Khovansky~\cite{Kho} found several bounds on topological invariants of varieties in terms of quantities measuring the complexity of their defining equations. In particular, he obtained fewnomial bounds for the number of nondegenerate positive solutions of polynomial systems, 
which were improved in 2007 by F. Sottile and F. Bihan~\cite{BS07}. But still the resulting bounds are not sharp in general.
\smallskip

Let us denote by $P_i$ the convex-hull of the support $\calW_i$ of the polynomial $f_i$. This is the  Newton polytope of $f_i$. The famous {\it Bernstein Theorem} (see~\cite{B}) bounds the number of nondegenerate complex solutions with nonzero coordinates of $f_1=\cdots=f_n=0$ by the {\it  mixed volume} of $P_1,\ldots,P_n$.
%
Write $P_I$ for the Minkowsky sum $\sum_{i \in I} P_i$, where $I$ is any nonempty subset
of $[n]=\{1,2,\ldots,n\}$. We have the classical formula for the mixed volume
\begin{equation}\label{E:mixedalternate}
\mbox{MV}(P_1,\dotsc,P_n)=\sum_{ \emptyset \neq I \subset {\bf [}n {\bf ]}} {(-1)}^{n-|I|} \mbox{Vol}(P_I),
\end{equation}
where $\mbox{Vol}(\cdot)$ is the usual euclidian volume in $\R^n$.
For instance, when $n=2$ we get $\mbox{MV}(P_1,P_2)=\mbox{Vol}(P_1+P_2)-\mbox{Vol}(P_1)-\mbox{Vol}(P_2)$.
%
%
Another probably less known formula involves the set of integer points of the polytopes $P_I$:
\begin{equation}\label{E:mixednumber}
\mbox{MV}(P_1,\dotsc,P_n)=\sum_{I \subset {\bf [}n {\bf ]}} {(-1)}^{n-|I|} |\Z^n \cap P_I|,
\end{equation} where the summation is taken over all subsets $I$ of $[n]$ including the emptyset
with the convention that $|\Z^n \cap P_{\emptyset}|=1$.
\smallskip

For any nonempty $I \subset [n]$, write $\calW_I$ for the set of points $\sum_{i \in I} w_i$ over all $w_i \in \calW_i$ with $i \in I$. Note that $P_I$ is the convex-hull of $\calW_I$. In view of \eqref{E:mixedalternate} and \eqref{E:mixednumber}, it is natural to expect that a generalized Descartes bound should depend not only on the numbers of elements of $\calW_1,\ldots,\calW_n$, but also on the number of elements of all sums $\calW_I$.
This leads us to define for {\it any number $r$} of finite sets $\calW_1,\ldots,\calW_r$ in $\R^n$ the associated {\it discrete mixed volume}
\begin{equation}\label{E:N}
D(\calW_1,\ldots,\calW_{r})=\sum_{I \subset {\bf [}r {\bf ]} } {(-1)}^{r-|I|}
\big| \calW_I \big|,
\end{equation}
where the sum is taken over all subsets $I$ of ${\bf [}r {\bf ]} $ including the empty set with the convention that $|\calW_{\emptyset}|=1$. Note that we do not impose that the sets $\calW_i$ are contained in $\Z^n$.
For instance, when $r=2$ we have $D(\calW_1,\calW_{2})=|\calW_1+\calW_{2}|-|\calW_1|-|\calW_{2}|+1$. 
%
%
%
%
We observe that $|\calW_1+\calW_{2}| \leq |\calW_1| \cdot |\calW_{2}|$, which readily implies that $D(\calW_1,\calW_{2})$ is bounded from above
by the Kouchnirenko number  $(|\calW_1|-1) (|\calW_2|-1)$. As a consequence, the above-mentionned counter-examples to the Kouchnirenko conjecture show that the
discrete mixed volume (with $r=n$) is not in general an upper bound
for the number of nondegenerate positive solutions of real polynomial systems with given supports.
However, we shall see that it provides bounds of this kind for
tropical polynomial systems
and real polynomial systems constructed by the combinatorial patchworking.
\smallskip

Tropical algebraic geometry is a recent area of mathematics which has led to several important results in different domains like real and complex geometry, enumerative geometry and combinatorics to cite only very few of them.
%
%
%
%
A tropical polynomial is a polynomial where the classical addition and multiplication are replaced with the maximum and the classical addition, respectively.
Thus a tropical polynomial is the maximum of a finite number of affine-linear functions. A tropical hypersurface in $\R^n$ is the corner locus of a tropical polynomial in $n$ variables.
This is a piecewise linear $(n-1)$-manifold in $\R^n$. A tropical polynomial also determines a convex polyhedral subdivision of its Newton polytope. This subdivision is {\it dual} to the tropical hypersurface via a bijection sending a polytope of positive dimension to a piece of complementary dimension and a vertex of the subdivision to a connected component of the complementary part of the tropical hypersurface.


%

The union of tropical hypersurfaces in $\R^n$ is dual to a mixed subdivision of the Minkowsky sum of their Newton polytopes. Tropical hypersurfaces intersect {\it transversely} at a point $p \in \R^n$ if the codimensions of the linear pieces
of the tropical hypersurfaces which contain $p$ in their relative interior sum up to the codimension
of their common intersection (see~\cite{BJSST07}) .
%
%
A transverse intersection point of tropical hypersurfaces will be called  {\it nondegenerate} by analogy with the classical case. Nondegenerate solutions of a system of $n$ tropical polynomial equations in $n$ variables are in one-to-one correspondence with the {\it mixed polytopes} of the dual mixed subdivision, which are $n$-dimensional Minkowsky sums of $n$ segments.
We show that the discrete mixed volume provides a fewnomial bound for the number of nondegenerate solutions of a tropical polynomial system.
\begin{thm}\label{T:Main} The number of nondegenerate solutions of a system of tropical polynomial equations with supports
$\calW_1,\ldots,\calW_n \subset \R^n$ does not exceed $D(\calW_1,\ldots,\calW_{n})$.
\end{thm}



To prove this result we develop a technique called {\it irrational mixed decomposition} which is mainly inspired by Chapter 7 of~\cite{CLO}.
We note that this irrational decomposition trick has also been used in the unmixed case in~\cite{BS}.
Given convex polytopes $P_1,\ldots,P_r \subset \R^n$ and finite subsets $\calW_I$ of the Minkowsky sums $P_I=\sum_{i \in I} P_i$, we consider the quantity
$N(\bW)=\sum_{I \subset {\bf [}r {\bf ]} } {(-1)}^{r-|I|} \big| \calW_I \big|$. For instance, if the polytopes are $n$ lattice polytopes in $\R^n$ and $\calW_I =\Z^n \cap P_I$, then
$N(\bW)$ coincides with the classical mixed volume of $P_1,\ldots,P_n$. The discrete mixed volume $D(\calW_1,\ldots,\calW_r)$ coincides with the quantity  $N(\bW)$
associated with the family $\bW=(\calW_I)$ defined by $\calW_I=\sum_{i \in I}\calW_i$.
Our irrational mixed decomposition allows us to estimate $N(\bW)$ under some assumptions on $\bW$. In particular, we are able to show the nonnegativity of $N(\bW)$ in some important cases (see Theorem \ref{T:nonnegative}), for instance,  when  $\bW$ is defined by $\calW_I =\Z^n \cap P_I$ for any number $r$ of lattice polytopes $P_1,\ldots,P_r \subset \R^n$, or if $N(\bW)=D(\calW_1,\ldots,\calW_r)$ for any number of sets $\calW_1,\ldots,\calW_r \subset \R^n$.
\begin{thm}\label{T:mainnonnegative}
We have the following nonnegativity properties.
\begin{enumerate}
\item
For any finite subsets $\calW_1,\ldots,\calW_r$ of $\R^n$, we have $D(\calW_1,\ldots,\calW_r) \geq 0$.
\item
If $P_1,\ldots,P_r$ are any lattice polytopes in $\R^n$ then
$$\sum_{I \subset {\bf [}r {\bf ]}} (-1)^{r-|I|} 
\big|\Z^n \cap P_I\big| \geq 0.$$
\end{enumerate}
\end{thm}

To our knowledge, Theorem \ref{T:mainnonnegative} is new except item (2) with $r=n$ or $r=n-1$ (see Section \ref{S:Mixedirrational}).

In the late 1970's Oleg Viro invented a powerful method to construct real algebraic varieties with prescribed topology. The Viro method can be seen as one of the root of tropical geometry.
Basically, one consider a real polynomial $f$ whose coefficients depend polynomially on a real parameter $t>0$. 
The smallest exponents of $t$ in the coefficients of $f$ define a convex subdivision of its Newton polytope.
The {\it combinatorial patchworking} arises when this subdivision is a triangulation whose set of
vertices coincides with the support of $f$. Viro patchworking Theorem asserts that for $t>0$ small enough the topological type of the hypersurface defined by $f$
in the real positive orthant can be read off the signs of the coefficients of the powers of $t$ with smallest exponents and the triangulation. In fact, one can see such a polynomial $f$
as a polynomial with coefficients in the field of real Puiseux series, and define the positive part of the corresponding tropical hypersurface taking into account of these signs. Then, the Viro patchworking Theorem can be rephrased by saying that for $t>0$ small enough the topological type of of the hypersurface defined by $f$
in the real positive orthant is homeomorphic to the positive part of the corresponding tropical hypersurface.
Similarly, one can also consider systems of $n$ polynomials in $n$ variables whose coefficients depend polynomially on $t$. The smallest exponents of $t$ in the coefficients of these polynomials define a convex mixed subdivision of the Minkowsky sum of their Newton polytopes. Assume that each individual convex subdivision is as above a triangulation whose set of vertices is the support of the corresponding polynomial. Assume furthermore
that the associated convex mixed subdivision is pure. Then, by a generalization of the Viro patchworking Theorem due to B. Sturmfels~\cite{St}, for $t>0$ small enough the number of solutions of the polynomial system
which are contained in the real positive orthant is equal to the number of nonempty mixed polytopes of the mixed subdivision. Recall that in this situation a mixed polytope is a Minkowsky sum of $n$ segments, and that the vertices of each segment correspond to a pair of monomials of one single polynomial in the system. Such a mixed polytope is called nonempty when the coefficients of each pair of monomials have oppositive signs. Fixing the individual supports $\calW_1,\ldots,\calW_n \subset \R^n$, and varying the smallest exponents of $t$ and the corresponding coefficient signs, one can look at the maximal number $T(\calW_1,\ldots,\calW_n)$ of nonempty mixed polytopes that can be obtained. From the tropical point of view, this quantity is the maximal number of nondegenerate positive solutions a tropical polynomial system with fixed supports $\calW_1,\ldots,\calW_n \subset \R^n$ can have. In 1996, I. Itenberg and M.-F. Roy~\cite{IR} made the conjecture that $T(\calW_1,\ldots,\calW_n)$ is the maximal number of nondegenerate positive solutions a real polynomial system with supports $\calW_1,\ldots,\calW_n$ can have. This was disproved in 1998  by Lee and Wang~\cite{LW}, but still the problem of finding an explicit upper bound for $T(\calW_1,\ldots,\calW_n)$ was open in general. Such an explicit bound is important since it gives the limits of the combinatorial patchworking method, which is one of the most powerful method for constructing polynomial systems with many positive solutions.
Such an explicit bound follows straightforwardly from Theorem \ref{T:Main}.
%
\begin{thm}\label{T:patch}
We have $T(\calW_1,\ldots,\calW_n) \leq D(\calW_1,\ldots,\calW_n)$. Thus
a real polynomial system with support $\calW_1,\ldots,\calW_n \subset \R^n$ and which is constructed by the combinatorial patchworking Viro method has at most
$D(\calW_1,\ldots,\calW_n)$ positive solutions.
\end{thm}

In fact Theorem \ref{T:patch} is true for a more general version of the Viro patchworking theorem arising when it is only assumed that the mixed subdivision is pure (this is the version obtained in ~\cite{Bi}).
We already observed that $D(\calW_1,\calW_2) \leq (|\calW_1|-1)(|\calW_2-1|)$ due to $|\calW_1+\calW_2| \leq |\calW_1| \cdot |\calW_2|$.
Similarly, if for any $I \subset [r]$ the sum $\sum_{I\in {\bf [}r {\bf ]}} \calW_i$ and the product $\prod_{I\in {\bf [}r {\bf ]}} \calW_i$ have the same number of elements,
then we immediately obtain $D(\calW_1,\ldots,\calW_r) =\prod_{i \in {\bf [}r {\bf ]}} (|\calW_i|-1)$. In particular, this shows that the discrete mixed volume and the Kouchnirenko number
coincide for sets in general position in $\R^n$. Using our irrational mixed decomposition technique, we can prove that $D(\calW_1,\ldots,\calW_r)$ cannot increase after small translations of the elements of $\calW_1,\ldots,\calW_r \subset \R^n$.

\begin{thm}\label{T:kouchbound} 
For any finite sets $\calW_1,\ldots,\calW_r \subset \R^n$, we have
$$D(\calW_1,\ldots,\calW_r) \leq \prod_{i \in {\bf [}r {\bf ]}} (|\calW_i|-1).$$
\end{thm}

Combined with Theorem \ref{T:Main}, this gives a sharp fewnomial bound for tropical polynomial systems.

\begin{thm}\label{T:MainBis} The Kouchnirenko number $\prod_{i \in {\bf [}n {\bf ]}} (|\calW_i|-1)$ is a sharp bound
for the number of nondegenerate solutions of a system of tropical polynomial equations with supports
$\calW_1,\ldots,\calW_n \subset \R^n$.
\end{thm}

Combining Theorem \ref{T:patch} and \ref{T:kouchbound} we obtain the following remarkable result, relating the Kouchnirenko conjecture to the conjecture proposed by Itenberg and Roy.

\begin{thm}\label{T:Kouchpatch} 
Kouchnirenko conjecture is true for polynomial systems constructed by the combinatorial patchworking Viro method, or equivalently, for tropical polynomial systems.
\end{thm}

The present paper is organized as follows. In Section 2 and Section 3 we recall basic facts about tropical geometry and the classical mixed volume.
Section 4 is devoted to our mixed irrational decomposition trick and contains some applications. The last section is concerned with the discrete mixed volume and fewnomial bounds for tropical polynomial systems.
  
We would like to thank Benjamin Nill for his interest in this work and useful discussions.

\section{Tropical hypersurfaces and combinatorial patchworking}

We recall basic facts about convex subdivisions and tropical hypersurfaces.
Standard references are for instance~\cite{Gathmann}, ~\cite{IMS}, ~\cite{Mtropappli},~\cite{RGST}.
We also adress the reader to~\cite{BB} for a detailed exposition with the same notations.
A polytope in $\R^n$ is the convex-hull of a finite number of points. It is called a {\it lattice polytope} if its vertices belong to $\Z^n$.
The {\it   lower part} of a polytope $P \subset \R^n$ with respect to a nonzero vector $\delta \in \R^n$ is the set $P_{\delta}$ defined as follows.
If $\dim P <n$, then $P_{\delta}=P$. If $\dim P =n$, then $P_{\delta}$ is the union of all (closed) facets of $P$ with inward normal vectors $n$ such that
$\langle n , \delta \rangle >0$.
\bigskip

\begin{figure}[htb]
\begin{center}
\includegraphics[scale=0.5]{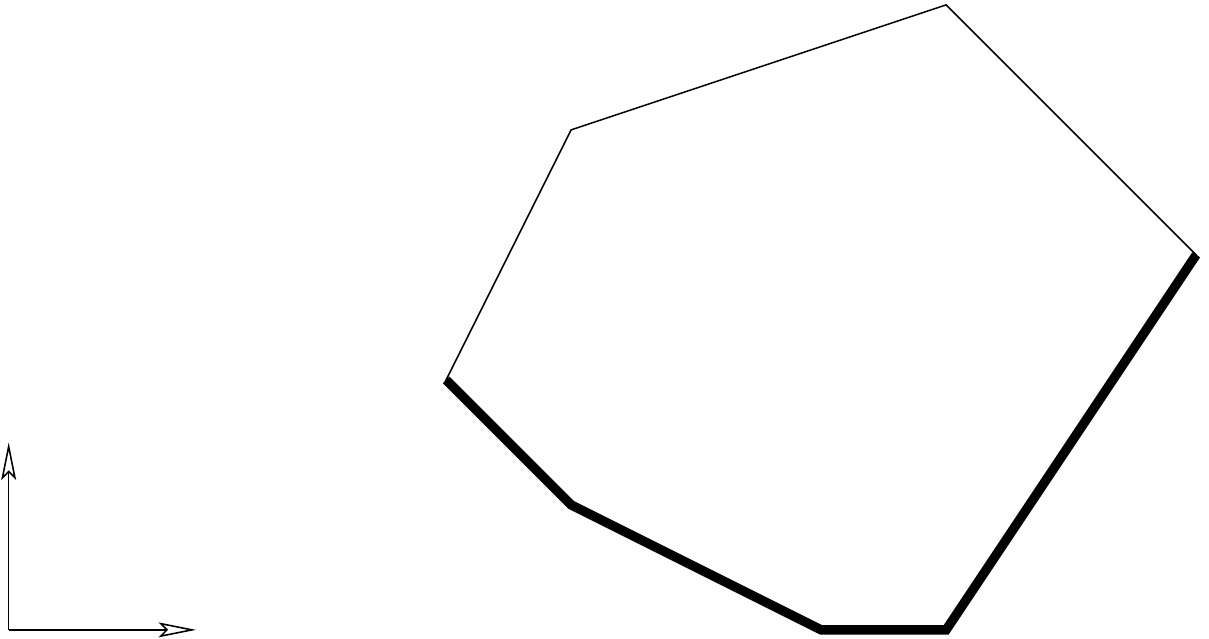}
\caption{Lower part of a polytope.}  
  \label{F:lowerpart}\end{center}
\end{figure}
\bigskip

Consider a finite set $\calW \subset \R^n$ and a map $\mu: \calW \rightarrow \R$. Let $P$ be the convex-hull of $\calW$ and
$\hat{P} \subset \R^{n+1}$ be the convex-hull of the set of points $(w,\mu(w))$ over all $w \in \calW$. The {\it lower part} of $\hat{P}$ is its lower part with respect
to the vertical vector $\delta=(0,\ldots,0,1)$. A {\it lower face} of $P$ is a face of $P$ contained in its lower part.
Projecting the lower faces of $\hat{P}$ onto $P$ via the projection  $\pi: \R^{n+1} \rightarrow \R^n$ forgetting the last coordinate,
we obtain a polyhedral subdivision of $P$ called {\it convex subdivision of $P$ associated with $\mu$}. 

%
%
%
%

\bigskip

\begin{figure}[htb]
\begin{center}
\includegraphics[scale=0.3]{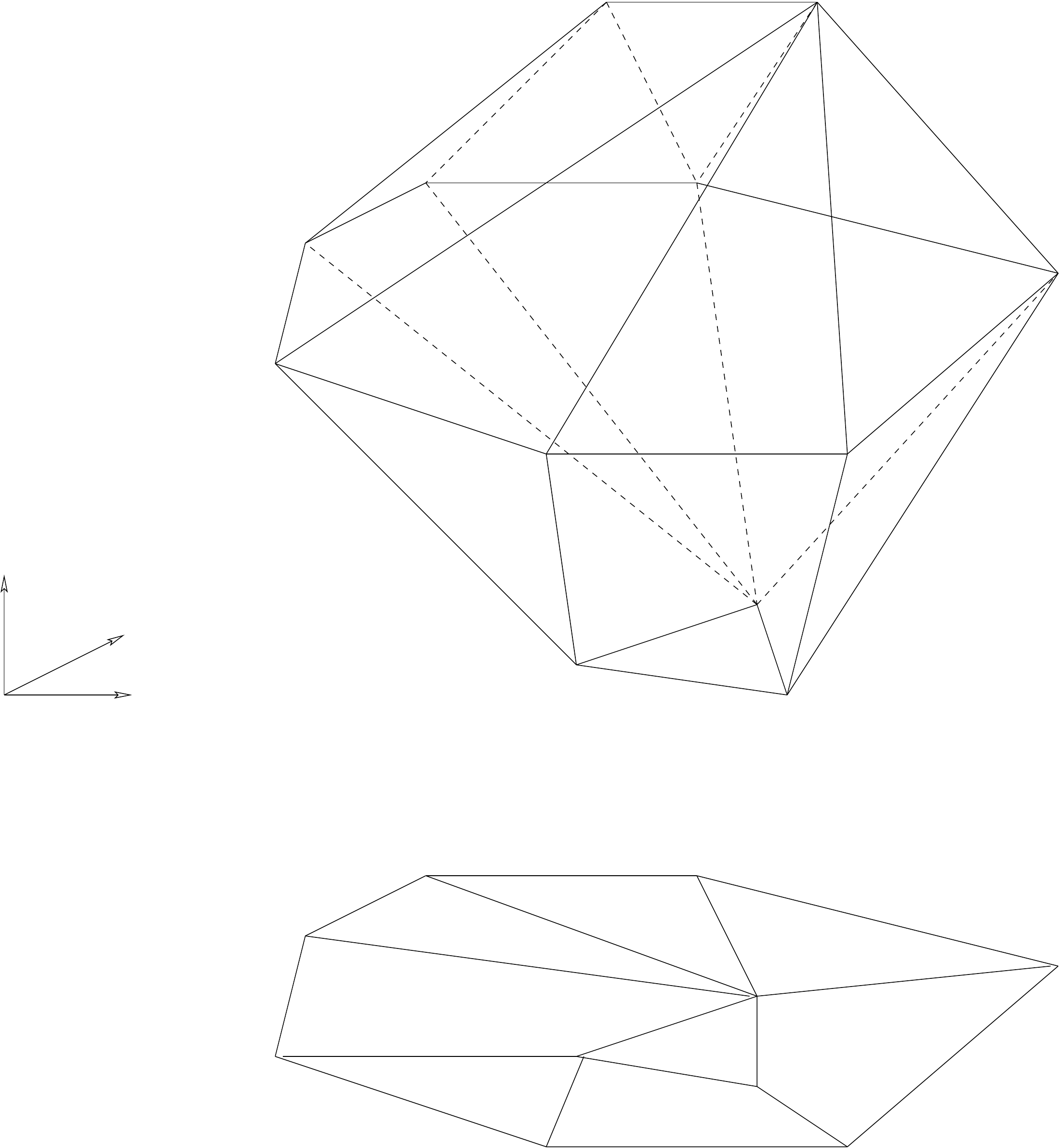}
\caption{A convex subdivision.
}  
  \label{F:lower3d}\end{center}
\end{figure}

Consider finite sets $\calW_1,\ldots,\calW_{r}$ in $\R^n$. The sum $\calW_1+\cdots+\calW_{r}$ is the set of points
$w_1+\cdots+w_{r}$ over all $(w_1,\ldots,w_r) \in  \calW_1 \times \cdots \times \calW_r$. Let $P_i$ be the convex-hull of $\calW_i$.
The Minkowsky sum $P=P_1+\cdots+P_r$ is the convex hull of $\calW=\calW_1+\cdots+\calW_{r}$.  
Consider now maps $\mu_i: \calW_i \rightarrow \R$ and denote by $\hat{P}_i$ the convex-hull of the set of points
$(w_i,\mu_i(w_i))$ over all $w_i \in \calW_i$. Let $\hat{P}$ be the Minkowsky sum
$\hat{P}_1+\cdots+\hat{P}_r$. Denote by $\calS_i$ the convex subdivision of $P_i$ associated with $\mu_i$.
Projecting the lower faces of $\hat{P}$ via $\pi$, we obtain a convex subdivision $\calS$ called {\it convex mixed subdivision} of $P$ associated with
$(\mu_1,\ldots,\mu_r)$. Any lower face $F$ of $\hat{P}$ can be uniquely written as a Minkowsky sum
of lower faces $F_1,\ldots,F_r$ of $\hat{P}_1,\ldots,\hat{P}_r$. Indeed, if $F$ is the largest face of $\hat{P}$ where the restriction to $\hat{P}$ of a linear function $\langle n , \cdot \rangle$
attains its minimum, then $F_i$ is the largest face of $\hat{P}_i$ where the same function attained its minimum on $\hat{P}_i$. Projecting onto $P$ via $\pi$, this induces a privileged writing
of any polytope $Q \in \calS$ as a Minkowsky sum $Q=Q_1+\cdots+Q_r$, where $Q_i \in \calS_i$ for $i=1,\ldots,r$. Note that such a writing is not unique in general, and we should always refer
to the previous privileged one when writing $Q=Q_1+\cdots+Q_r$ with $Q \in \calS$ and $Q_i \in \calS_i$ for $i=1,\ldots,r$.

\begin{figure}[htb]
\begin{center}
\includegraphics[scale=0.5]{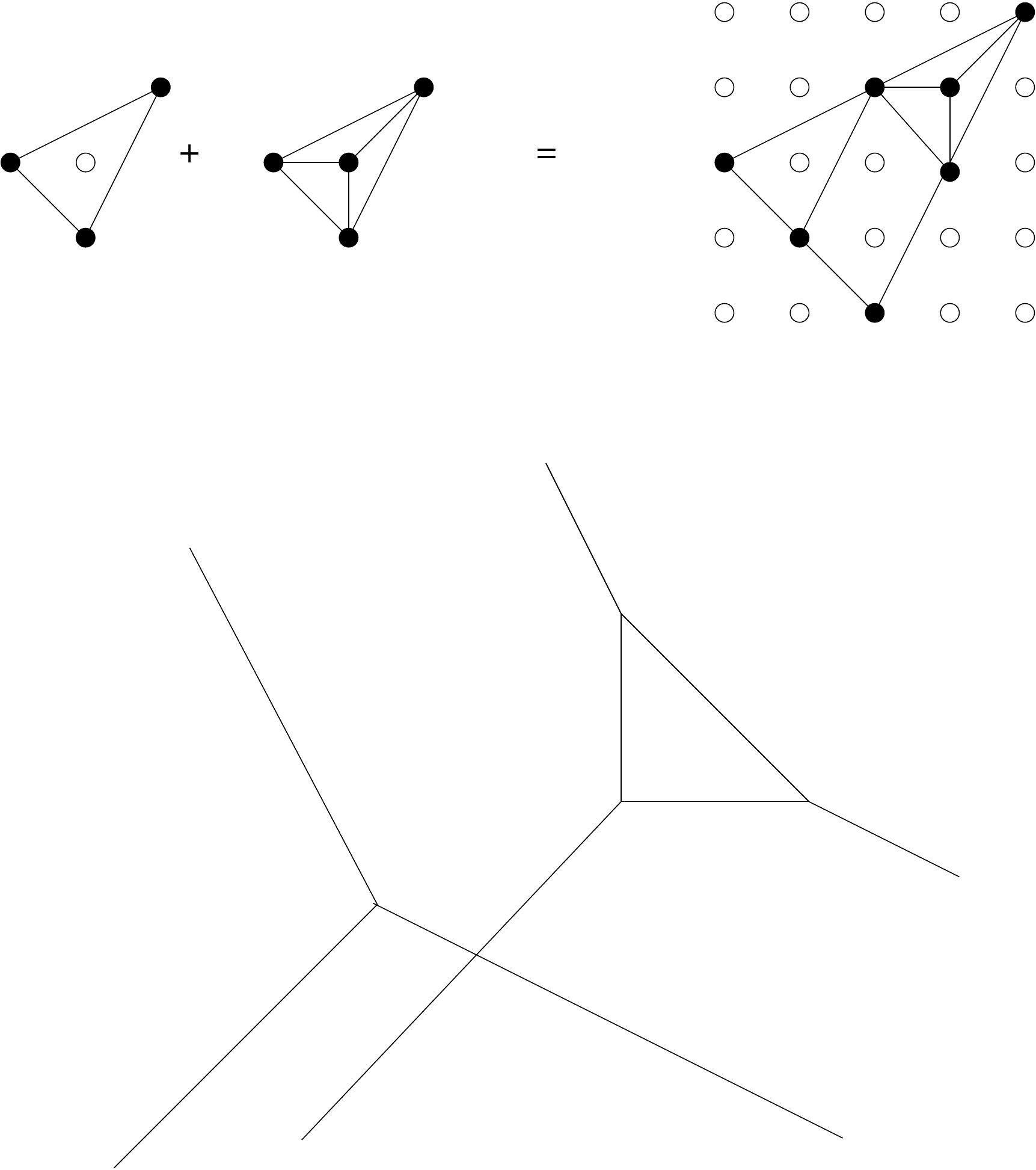}
\caption{A pure convex mixed subdivision and the corresponding arrangement of tropical curves.}  
  \label{F:convex mixed}\end{center}
\end{figure}

A tropical polynomial is a polynomial $\sum_{w \in \calW} a_wx^w$ with real coefficients, where the classical addition and multiplication are replaced with the maximum and the classical addition, respectively. Thus a tropical polynomial is the maximum of a finite number of affine-linear functions. A tropical hypersurface in $\R^n$ is the corner locus of a tropical polynomial in $n$ variables.
It is a piecewise linear $(n-1)$-manifold in $\R^n$, whose pieces together with the connected components of the complementary part form a subdivision of $\R^n$.
This subdivision is dual to the convex polyhedral subdivision $\calS$ of the convex hull $P$ of $\calW$ associated with the map $\mu: \calW \rightarrow \R$, $w \mapsto -a_w$.

This duality is a bijection which reverses the dimensions and sends a piece of the tropical hypersurface to a polytope contained in some orthogonal affine space.
Note that some monomial of a tropical polynomial may not contribute when taking the maximum of the corresponding linear functions. In fact,
the monomials which contribute correspond precisely to the vertices of the subdivision dual to the associated tropical hypersurface, so that this set of vertices can be thought of as the
support of the tropical polynomial.
%

Consider tropical hypersurfaces in $\R^n$ defined by tropical polynomials $f_1,\ldots,f_r$ with newton polytopes $P_1,\ldots,P_r$.
Each polynomial $f_i$ gives rise a to a subdivision of $\R^n$ dual to the convex subdivision $\calS_i$ of $P_i$ associated with a map $\mu_i$. The pieces of
the union of the tropical hypersurfaces defined by $f_1,\ldots,f_r$ together with the connected components of the complementary part of this union
form a subdivision of $\R^n$. This subdivision is dual in the previous sense to the convex mixed subdivision $\calS$ of the Minkowsky sum $P=P_1+\cdots+P_r$ associated
with the maps $\mu_1,\ldots,\mu_r$. As before, this duality is a bijection which reverses the dimension and sends a piece to a polytope contained in some
orthogonal affine space. If $\xi$ is a piece dual to a polytope $Q \in \calS$ such that $Q=Q_1+\cdots+Q_r$ with $Q_i \in \calS_i$ for $i=1,\ldots,r$ (this is the privileged writing described above),
then $\xi=\cap_{i=1}^r \xi_i$ where $\xi_i$ is the piece of the subdivision of $\R^n$ induced by $f_i$ which is dual to $Q_i$. Moreover, these pieces $\xi_i$ are the smallest
such that $\xi=\cap_{i=1}^r \xi_i$ (in other words $\xi$ is contained in the relative interior of each $\xi_i$).
Note that the connected components of complementary part of the union of the tropical hypersurfaces are dual to polytopes
$Q=Q_1+\cdots+Q_r \in \calS$ such that at least one $Q_i$ is a vertex of the subdivision $\calS_i$. The tropical hypersurfaces defined by $f_1,\ldots,f_r$
intersect transversely at a common intersection point $p$ if the codimensions of the linear pieces $\xi_1,\ldots,\xi_r$
of the tropical hypersurfaces which contain $p$ in their relative interior sum up to the codimension
of their common intersection. This is equivalent to $\dim Q=\dim Q_1+ \cdots+ \dim Q_r$ for the polytope $Q=Q_1+\cdots+Q_r  \in \calS$
dual do the piece $\xi$ containing $p$ in its relative interior. This is the usual notion of transversality for tropical hypersurfaces, see~\cite{BJSST07} for instance.
Therefore, transversal common intersection points of the tropical hypersurfaces defined by $f_1,\ldots,f_r$ are in bijection with the polytopes
$Q=Q_1+\cdots+Q_r \in \calS$ such that $\dim Q=\dim Q_1+ \cdots+ \dim Q_r$ and $\dim Q_i \geq 1$ for $i=1,\ldots,r$. Such a polytope $Q$ is called mixed polytope.
When $r=n$, a mixed polytope is an $n$-dimensional polytope $Q=Q_1+\cdots+Q_n \in \calS$ such that each $Q_i$ is a segment,
and thus nondegenerate solutions of a system of $n$ tropical polynomial equations in $n$ variables are in bijection with these mixed polytopes.

A mixed subdivision $\calS$ is called {\it pure} when $\dim Q=\sum_{i=1}^r \dim Q_i$ for each $Q=Q_1+\cdots+Q_r\in \calS$,
including non mixed polytopes $Q$ (for which some summand $Q_i$ is a point). This means that any common intersection of some tropical hypersurfaces
among those defined by $f_1,\ldots,f_r$ is a transversal intersection point.
%

Consider the field $\KK$ of complex Puiseux series. Denote by $Val$ the map sending a Puiseux series to minus its valuation et extend with respect to each coordinate
to a map $\KK^n \rightarrow \R^n$. For any $g =\sum_{w \in \calW} s_{w}z^{w} \in \KK[z]$, define the associated tropical polynomial $f=Trop(g)$ by $f(x)=\sum_{w \in \calW} Val(s_{w})x^{w}$. Kapranov's theorem asserts
that the image by $Val$ of the hypersurface defined by $g$ is the tropical hypersurface defined by $f=Trop(g)$. Moreover,
this tropical hypersurface is dual to the convex subdivision of the convex-hull $P$ of $\calW$ associated with the map $\mu=-Val$.
Consider now polynomials $g_1,\ldots,g_n$ in $n$ variables with coefficients in $\KK$. Set $f_i=Trop(g_i)$, let $P_i$ be the Newton polytope of $f_i$ and denote as above by $\calS_i$
the associated convex subdivision of $P_i$. Assume that the tropical hypersurfaces defined by $f_1,\ldots,f_n$ intersect transversally, which means that the dual convex mixed subdivision $\calS$ of $P=P_1+\cdots+P_n$ is pure. Assume that the system $g_1=\cdots=g_n$ has a finite number of solutions in $(\KK^*)^n$. Then, the number of such solutions counted with multiplicities  which project to a given solution of the tropical polynomial system $f_1=\cdots=f_n=0$ is equal to the euclidian volume of the dual mixed polytope.
This implies that the total number of solutions of the system $g_1=\cdots=g_n$ in $(\KK^*)^n$ is equal (when it is finite) to the mixed volume of the Newton polytopes
$P_1,\ldots,P_n$.

\begin{figure}[htb]
\begin{center}
\includegraphics[scale=0.5]{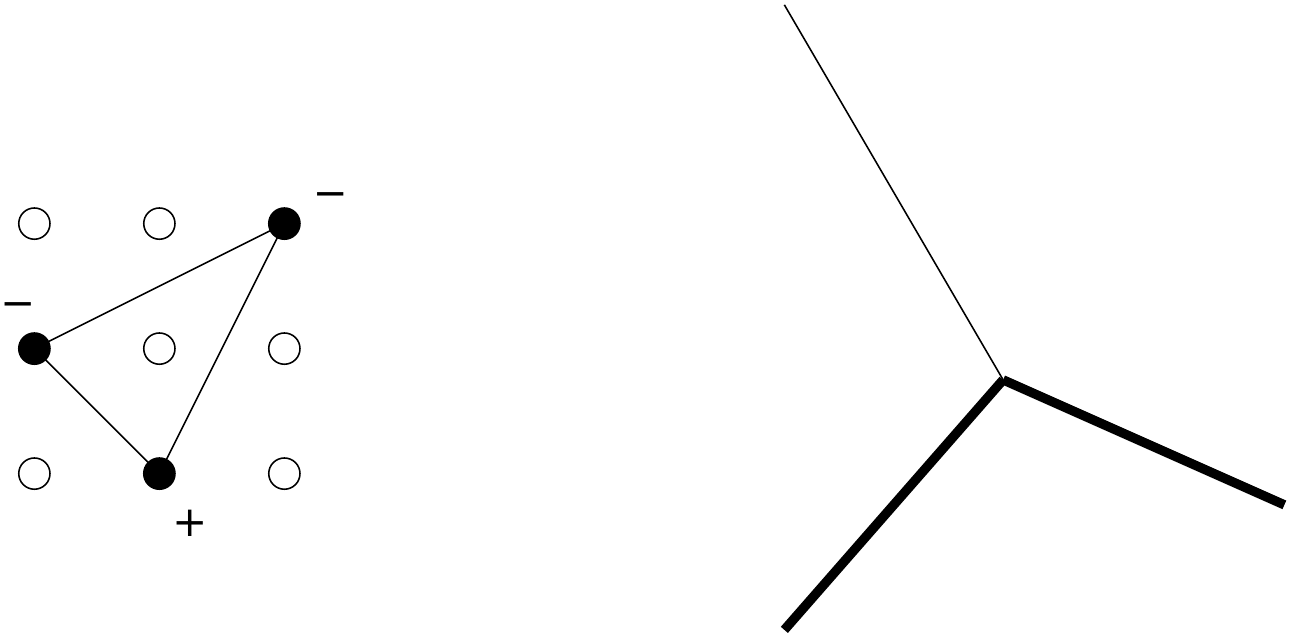}
\caption{A plane tropical curve, its positive part in bold and dual convex subdivisions with vertices equipped with signs.}  
  \label{F:individual curve1}\end{center}
\end{figure}
\bigskip

\begin{figure}[htb]
\begin{center}
\includegraphics[scale=0.5]{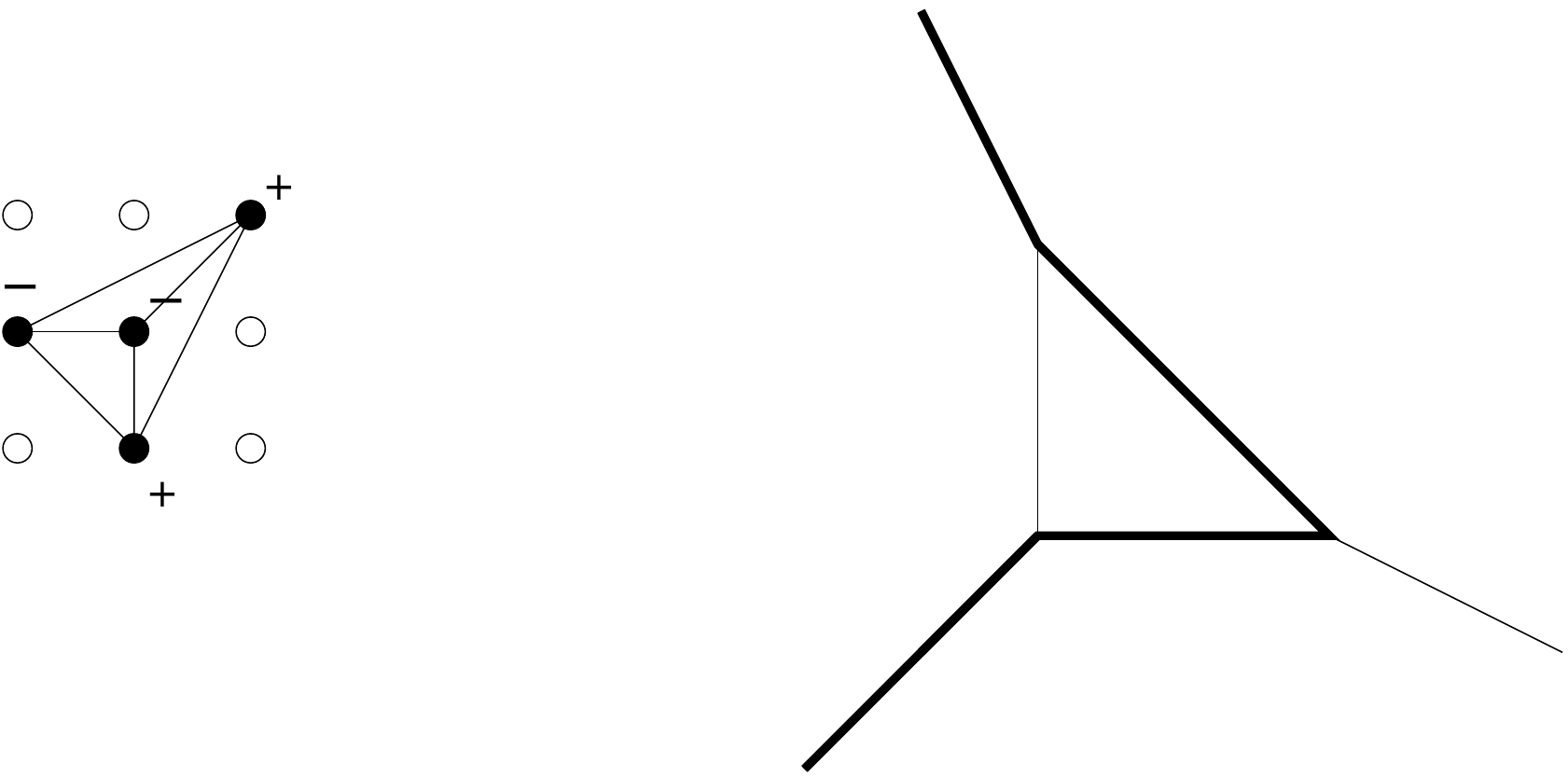}
\caption{Another plane tropical curve, its positive part in bold and dual convex subdivisions with vertices equipped with signs.}  
  \label{F:individual curve2}\end{center}
\end{figure}
\bigskip

Assume now that each polynomial $g_i$ has coefficients in the field of real Puiseux series.
Consider a vertex $w \in \calS_i$ and let $s(w)$ be the coefficient of the corresponding monomial of $g_i$.
Hence $s(w)$ is a real Puiseux series in the variable $t$, with some valuation $v$, $s(w)=at^{v}$ plus higher order terms ($v=-Val(s(w))$ with our notations).
Define the sign (plus or minus) of $w$ as the sign of the coefficient of $s(w)$ corresponding to the smallest exponent:
if $s(w)=at^{v}$ plus higher order terms, so that $Val(s(w))=-v$, then the sign of $s(w)$ is $+$ if $a>0$ and $-$ otherwise.
A segment $Q_i \in \calS_i$ is called {\it nonempty} if its enpoints have opposite signs.
Define the {\it positive part} of the tropical hypersurface defined by $f_i$ as the union of all pieces of the tropical hypersurface which are dual to nonempty segments
in $\calS_i$. Then the common intersection point of the positive parts of the tropical hypersurfaces defined by $f_1,\ldots,f_n$ are dual to the 
mixed polytopes $Q=Q_1+\cdots+Q_n \in \calS$ such that each segment $Q_i$ is non-empty. Such a mixed polytope is called {\it nonempty}
(see Figure \ref{F:nonemptymixedpolytope}).

\begin{figure}[htb]
\begin{center}
\includegraphics[scale=0.5]{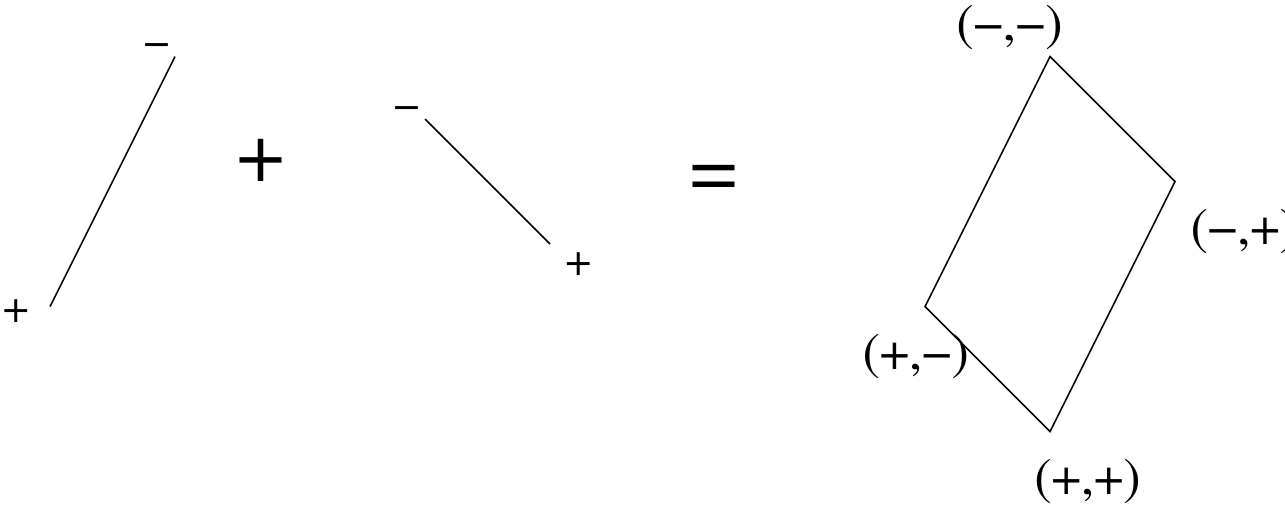}
\caption{A nonempty mixed polytope for $r=n=2$.}  
  \label{F:nonemptymixedpolytope}
\end{center}
\end{figure}
\bigskip

The combinatorial patchworking theorem for zero-dimensional varieties obtained in~\cite{St} can be described as follows. 
Consider a polynomial system $g_{1,t}(x)=\cdots=g_{n,t}(x)=0$ defined by Laurent polynomial in $n$ variables $x=(x_1,\ldots,x_n)$ of the form $g_{i,t}(x)=\sum_{w \in \calW_i}c_{i,w}t^{\mu_i(w)}x^{w}$,  where $t$ is some positive parameter and $\mu_i : \calW_i \rightarrow \R$.
We may see each polynomial $g_{i,t}$ as a polynomial with coefficient in the field of real Puiseux series in $t$, and consider the corresponding tropical polynomial system
$f_1=\cdots=f_n=0$ as above. 
Assume that each convex subvision $\calS_i$ is a triangulation with set of vertices $\calW_i$, and that the convex mixed subdivision $\calS$ is pure. Then, the combinatorial patchworking theorem
asserts that for $t>0$ small enough the number of nondegenerate positive solutions of the system $g_{1,t}(x)=\cdots=g_{n,t}(x)=0$ coincides with the number of positive solutions of the tropical polynomial system $f_1=\cdots=f_n=0$, and thus to the number of nonempty mixed polytopes of $\calS$.

\begin{figure}[htb]
\begin{center}
\includegraphics[scale=0.5]{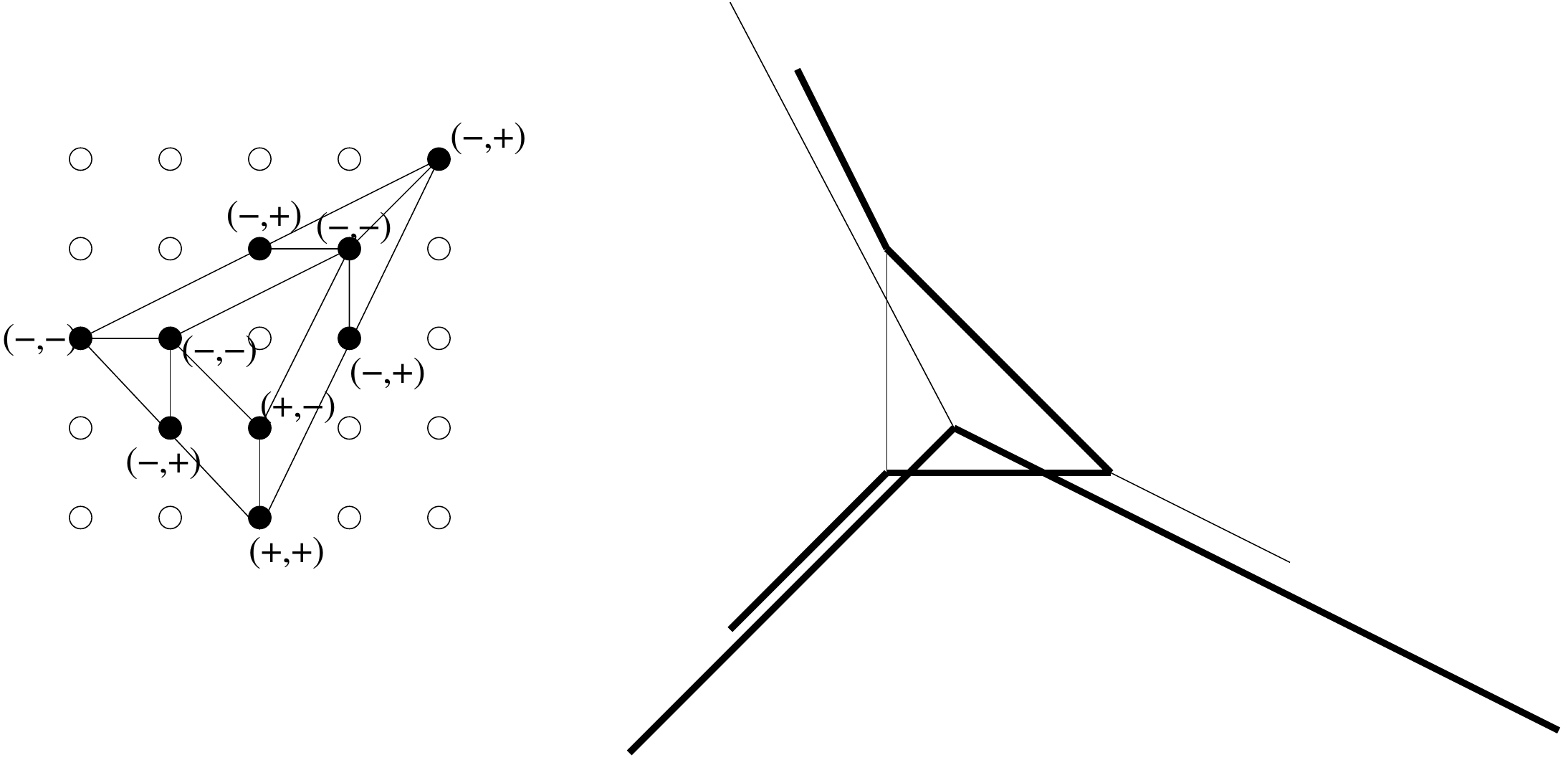}
\caption{Arrangement of two plane tropical curves, their positive parts and the dual mixed subdivision.}  
  \label{F:arrangement1}\end{center}
\end{figure}
\bigskip

\begin{figure}[htb]
\begin{center}
\includegraphics[scale=0.5]{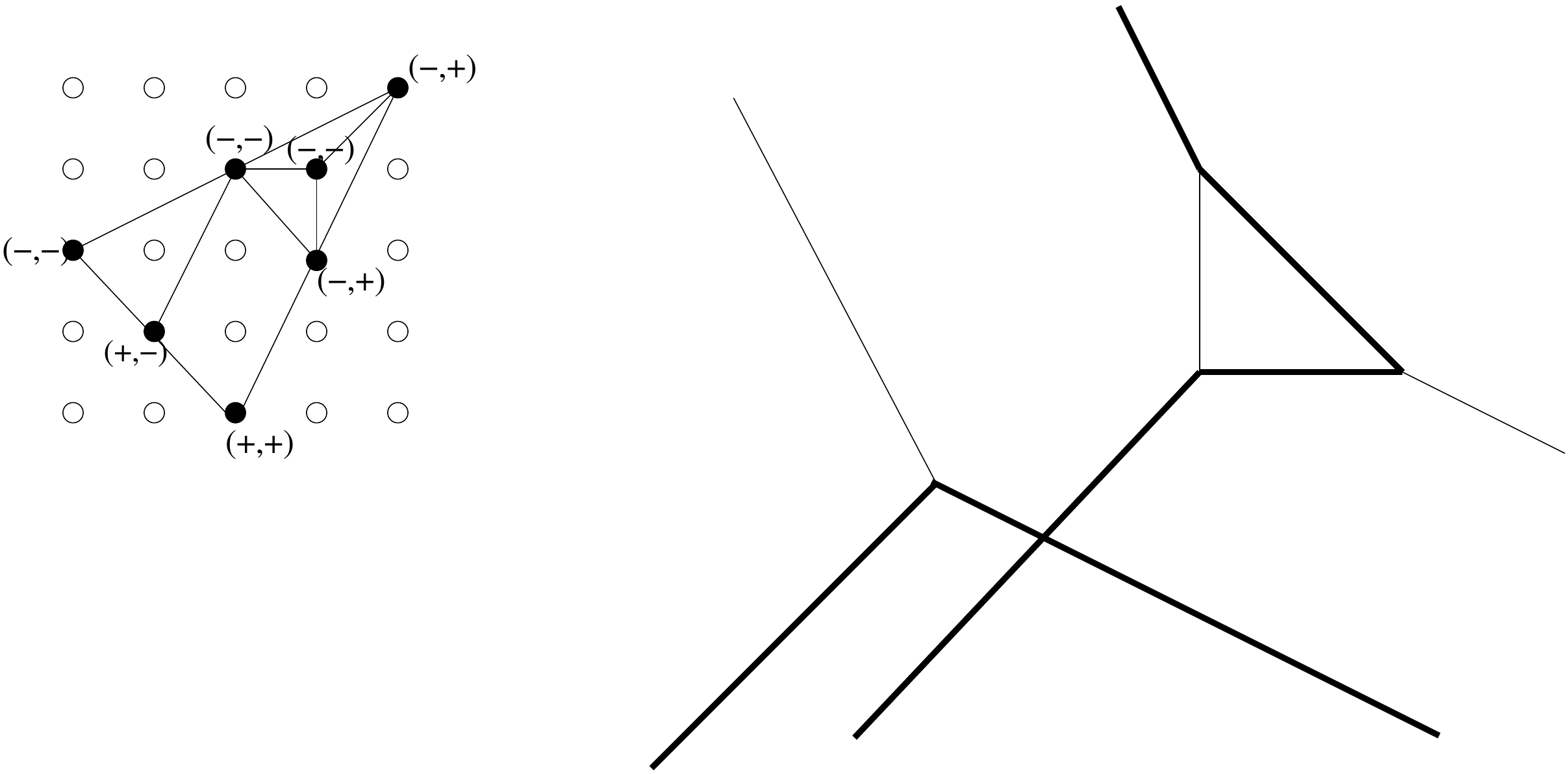}
\caption{Another arrangement of plane tropical curves together with the positive parts and the dual mixed subdivision.}  
  \label{F:arrangement2}\end{center}
\end{figure}

\medskip

As an example, consider the case $n=r=2$ given by the intersection of two plane tropical curves.
In Figures \ref{F:individual curve1} and \ref{F:individual curve2}, we have depicted two plane tropical curves, their positive parts in bold,
together with the dual individual pure convex subdivisions and vertices equipped with signs.
Figure \ref{F:arrangement1} and Figure \ref{F:arrangement2} provide examples of two different arrangements of two plane tropical curves.
In each arrangement, both tropical curves have the dual convex subdivisions with vertices equipped with signs from
Figures \ref{F:individual curve1} and \ref{F:individual curve2}.
The dual convex mixed subdivisions $\calS$ corresponding to each arrangement is given.
Moreover, each vertex $v=v_1+v_2 \in \calS$ is labelled with a pair of signs $(s_1,s_2)$,
where $s_i$ is the sign of the vertex $v_i \in \calS_i$. In Figure \ref{F:arrangement1} we get $3$ transversal intersections points, two of them are positive points, while
in Figure \ref{F:arrangement2} we get one transversal intersection point which is a positive point.

\section{Classical mixed volume}
\label{S:classical mixed volume}

We recall here basic results about the classical mixed volume which will be useful later.
Consider polytopes $P_1,\ldots,P_n$ in $\R^n$. Assume that $P=P_1+\cdots+P_n$ has dimension $n$.
For any non negative real numbers $\lambda_1,\ldots,\lambda_n$, the function
$(\lambda_1,\ldots,\lambda_n) \mapsto \mbox{Vol}_n(\lambda_1P_1+\cdots+\lambda_nP_n)$ is homogeneous of degree $n$ (without the assumption $\dim P=n$, this map is
homogeneous of degree $\dim P$). The mixed volume $\mbox{MV}(P_1,\ldots,P_n)$ is the coefficient of the monomial $\lambda_1 \cdots\lambda_n$ in this polynomial.
We have the classical result (see~\cite{CLO} for instance).

\begin{prop}\label{P:brick}
Let $\calW_1,\ldots,\calW_{n}$ be finite subsets of $\Z^n$ with convex-hulls $Q_1,\ldots,Q_n$, respectively. Assume that each $Q_i$ is a segment and that
$\dim Q= \dim Q_1+\cdots+\dim Q_n$, where $Q=Q_1+\cdots+Q_n$.
Then, for any sufficiently small vector $\delta \in \R^n$ not parallel to a face of  $Q$, we have
\begin{equation}\label{E:MVmixed polytope}
\mbox{MV}(Q_1,\ldots,Q_n)=\mbox{Vol}_n(Q)= |\Z^n \cap (\delta+Q)|.
\end{equation}
%
\end{prop}

%

With the help of a convex pure mixed subdivision of $P=P_1+\cdots+P_n$, we may compute $\mbox{MV}(P_1,\ldots,P_n)$ by means of the following classical result.

\begin{prop}\label{P:mixedvolumesum}
Let $\calS$ be any convex pure mixed subdivision of $P=P_1+\cdots+P_{n} \subset \R^n$.
Then
$$
\mbox{MV}_{n}(P_1,\dotsc,P_{n})=\sum_
{Q \; \mbox{\tiny{mixed polytope of}} \; \calS}
\mbox{Vol}_{n}(Q)
$$
\end{prop}

A proof of Proposition \ref{P:mixedvolumesum} goes as follows. Assume that the convex pure mixed subdivision $\calS$ is associated with maps $\mu_1,\ldots,\mu_n$,
where $\mu_i$ is some real-valued map defined on a finite set $\calW_i$ with convex-hull $P_i$. Then, for any nonnegative real number $\lambda_1,\ldots,\lambda_n$,
the convex mixed subdivision of $\lambda P_1+\cdots+\lambda_n P_n$ associated with the maps $\lambda_i \calW_i \rightarrow \R$, $\lambda_i x \mapsto \lambda_i \mu_i(x)$,
consists of all polytopes $\lambda_1Q_1+\cdots+\lambda_nQ_n$ with $Q=Q_1+\cdots+Q_n \in \calS$. We obviously have
$$
\begin{array}{lll}
\mbox{Vol}_n(\lambda_1P_1+\cdots+\lambda_nP_n) & = & \sum_{Q \in \calS} \mbox{Vol}_n(\lambda_1Q_1+\cdots+\lambda_nQ_n) \\
& & \\
 & = & \sum_{Q \in \calS}\mbox{Vol}_n(Q_1+\cdots+Q_n) \cdot \lambda_1^{\dim(Q_1)}\cdots\lambda_n^{\dim(Q_n)},
\end{array}$$
the latter equality coming from the fact that $\calS$ is pure. The coefficient of $\lambda_1 \cdots \lambda_n$
in the last expression is precisely the total volume of the mixed polytopes of $\calS$.

As it is well-known, Proposition \ref{P:mixedvolumesum} and its proof can be generalized to any number of polytopes as follows.
Let $P_1,\ldots,P_r$ be polytopes in $\R^n$ such that $P=P_1+\cdots+P_r$ has dimension $n$. Let $\calS$ be a convex pure mixed subdivision of $P$.
The function $(\lambda_1, \ldots, \lambda_r) \mapsto  \mbox{Vol}_n(\lambda_1P_1+\cdots+\lambda_rP_r)$ defined for any non negative real numbers $\lambda_1,\ldots,\lambda_r$ is
homogeneous of degree $n$. The coefficient of $\lambda_1^{a_1}\cdots \lambda_r^{a_r}$ is equal to the sum $\sum \mbox{Vol}_{n}(Q)$ over all polytopes $Q=Q_1+\cdots+Q_r \in \calS$ such that
$\dim Q_i=a_i$ for $i=1,\ldots,r$. Moreover, this coefficient is equal to $\frac{1}{a_1 ! \cdots a_r !} \cdot \mbox{MV}_{n}(P_1,\ldots,P_1,\dotsc,P_r,\ldots,P_{r})$,
where $P_i$ is repeated $a_i$ times.
\medskip

We have the following alternative definition of the mixed volume.
\begin{prop}
\label{P:alter}
We have
$$
\mbox{MV}_n(P_1,\dotsc,P_n)=\sum_{ \emptyset \neq I \subset {\bf [}n {\bf ]}} {(-1)}^{n-|I|} \mbox{Vol}_n(\sum_{i \in I} P_i).
$$
\end{prop}

A proof of  Proposition \ref{P:alter} can be obtained using Proposition \ref{P:mixedvolumesum} and the inclusion-exclusion principle.
We present such a proof as a warm up for the proof of Proposition \ref{P:Nsum} in Section \ref{S:Mixedirrational}.
Consider any pure convex mixed subdivision $\calS$ of $P$.
We have $\mbox{Vol}_{n}(P)=\sum_ {Q \in \calS} \mbox{Vol}_{n}(Q)$,
where it suffices to take the sum over all $n$-dimensional polytopes $ Q \in \calS$. Since $\calS$ is pure, any $n$-dimensional polytope $Q =Q_1+\cdots+ Q_n \in \calS$
which is not a mixed polytope should have at least one zero-dimensional summand $Q_i$. Therefore,
$$
\mbox{MV}_{n}(P_1,\dotsc,P_{n})=\sum_{Q\; \mbox{\tiny mixed polytope of} \; \calS}  \mbox{Vol}_{n}(Q) = \mbox{Vol}_{n}(P) 
-\sum_{Q \in \calS, \; \exists i \in {\bf [}n {\bf ]} ,\; \dim Q_i=0}\mbox{Vol}_{n}(Q)
$$
and thus by inclusion-exclusion principle
$$
\mbox{MV}_{n}(P_1,\dotsc,P_{n})=
\mbox{Vol}_{n}(P)+
\sum_{I \varsubsetneq {\bf [}n {\bf ]}} (-1)^{n-|I|} \sum_{\tiny Q \in \calS, \; \forall \,  i \notin I, \; \dim Q_i=0} 
\mbox{Vol}_{n}(Q)
.$$
It remains to see that for any $I \subsetneq {\bf [}n {\bf ]}$ we have $$\sum_{\tiny Q \in \calS, \; \forall \,  i \notin I, \; \dim Q_i=0} 
\mbox{Vol}_{n}(Q)=\mbox{Vol}_{n}(\sum_{i \in I} P_i).$$

This is a consequence of the following easy result, which will be also used in Section \ref{S:Mixedirrational}.

\begin{lemma}
\label{L:keylemma}
Let $P_1,\ldots,P_r \subset \R^n$ be polytopes in $\R^n$ with Minkowsky sum $P$ of dimension $m$.
Let $\calS$ be a convex pure mixed subdivision of $P$ associated with function $\mu_1,\ldots,\mu_r$. For any nonempty $I \subset [r]$, denote
by $\calS_I$ the convex pure mixed subdivision of $P_I=\sum_{i \in I} P_i$.
For any nonempty $I \varsubsetneq [r]$ and any $m$-dimensional polytope $Q_I \in \calS_I$, there exists an unique vertex $v_{Q_I} \in \calS_{I^c}$, where $I^c=
[r] \setminus I$, such that $Q=Q_I+v_{Q_I}$ is an $m$-dimensional polytope of $\calS$. Conversely, any $m$-dimensional polytope $Q=Q_1+\cdots+Q_r \in \calS$
such that $\dim Q_i=0$ for all $i \notin I$ arises in this way.
\end{lemma}
\begin{proof}
Without loss of generality, we may assume $m=n$. Any $n$-dimensional polytope $Q_I \in \calS_I$
is the image by the vertical projection of a lower $n$-dimensional facet of $\sum_{i \in I} \hat{P}_i$. Let $u \in \R^{n+1}$ such that the restriction
to $\sum_{i \in I} \hat{P}_i$ of the linear map $\langle u, \cdot \rangle$ takes its minimal value on this facet. For $i \notin I$, let $\hat{Q}_i$ be the face of
$\hat{P}_i$ where the restriction of the previous linear map takes its minimal value. Let $Q_i \in \calS_i$ be the image of $\hat{Q}_i$ under the vertical projection.
Then  $Q_I+\sum_{i \notin I} Q_i$ belongs to $\calS$. But since $\calS$ is pure and $\dim Q_I=n$, we should have $\dim Q_i=0$ for all $i \notin I$. The converse is obvious.
\end{proof}

\section{Mixed irrational decomposition}
\label{S:Mixedirrational}

Let $P_1,\ldots,P_r$ be polytopes  in $\R^n$ with Minkowsky sum $P$. Denote by $m$ the dimension of $P$.
Consider functions $\mu_i: \calV_i \rightarrow \R$, $i=1,\ldots,r$, where $\calV_i$ is a finite set with convex-hull $P_i$.
Assume that the convex mixed subdivision $\calS$ of $P$ associated with $(\mu_1,\ldots,\mu_r)$ is pure.
Consider the pure mixed subdivision $\calS_I$ of $P_I=\sum_{i \in I}P_i$ associated with $\mu_I$ with $i \in I$ and denote by $\calV_I$ its set of vertices.
We may assume for simplicity that the convex subdivision of $P_i$ associated with $\mu_i$ is $\calV_i$, so that $\calV_{\{i\}}=\calV_i$.
\smallskip

Suppose that for any nonempty $I \subset [r]$, a finite
subset $\calW_I$ of $P_I$ is given. We set $\calW=\calW_{[r]}$ and $\calW_i=\calW_{\{i\}}$ for $i=1,\ldots,r$.
By Lemma \ref{L:keylemma}, for any nonempty $I \varsubsetneq [r]$, and any $m$-dimensional polytope $Q_I \in \calS_I$, there exists an unique vertex $v_{Q_I} \in \calS_{I^c}$, where $I^c=
[r] \setminus I$, such that $Q_I+v_{Q_I}$ is an $m$-dimensional polytope of the subdivision $\calS$ of $P$.
We say that the family $\bW=(\calW_I)_{I \subset [r]}$ {\it satisfies property (S) with respect to $(\mu_1,\ldots,\mu_r)$} if
for any such $Q_I$ we have
$$
(\calW_I \cap Q_I) + v_{Q_I} \subset \calW.$$
(If $\calW_I \cap Q_I = \emptyset$, this is an empty condition.)
Recall that $\calW=\calW_{[r]}$.
We simply say that $\bW$ satisfies property $(S)$ if there are functions $\mu_1,\ldots,\mu_r$ as above such that
$\bW$ satisfies property (S) with respect to these functions.

Here are our main examples of families satisfying property (S). Note that in all these examples, but the last one, the family satisfies the stronger property
that for any nonempty $I \varsubsetneq [r]$, any $J \subset [r] \setminus I$ and any points $v_j \in \calV_j$ with $j \in J$ we have
$$\calW_I+\sum_{j \in J} v_j \subset \calW_{I \cup J}.$$

\begin{ex}\label{E:main} {\it (Families satisfying property (S).)}
\medskip

\begin{enumerate}
\item 
For $i=1,\ldots,r$, let $\calW_i$ be any finite nonempty set.
For any nonempty $I \subset [r]$, set $\calW_I= \sum_{i \in I} \calW_i$.
Then $\bW=(\calW_I)_{I \subset [r]}$ satisfies property (S) with respect to any sufficiently generic functions $\mu_i: \calW_i \rightarrow \R$.
\smallskip

\item 
For any nonempty $I \subset [r]$, set $\calW_I=\Z^n \cap P_I$. Then $\bW=(\calW_I)_{I \subset [r]}$ satisfies property (S) with respect to any sufficiently generic functions $\mu_i: \calV_i \rightarrow \R$ such that $\calV_i$ is a  set of lattice points with convex hull $P_i$.
\smallskip

\item 
For $i=1,\ldots,r$ let $\calW_i$ be any finite nonempty set with convex hull $P_i$. For any nonempty $I \subset [r]$, define $\calW_I$ as the intersection of $\sum_{i \in I} \calW_i$
with the absolute interior of $P_I$. Then $\bW=(\calW_I)_{I \subset [r]}$ satisfies property (S) with respect to functions as in item (1).
\smallskip

\item 
Assume that $P_1,\ldots,P_r$ are lattice polytopes. For any nonempty $I \subset [r]$, define $\calW_I$ as the set of lattice points contained in the absolute interior of $P_I$.
Then $\bW=(\calW_I)_{I \subset [r]}$ satisfies property (S) with respect to functions as in item (2).

\item 
The family $\bV=(\calV_I)_{I \subset [r]}$, where as above $\calV_I$ is the set of vertices of the pure convex subdivisions of $P_I$ induced by $\mu_i$ for $i \in I$ obviously satisfies property (S)
with respect to $\mu_1,\ldots,\mu_r$.
\end{enumerate}
\end{ex}

Note that taking the relative interior instead of the absolute interior in items (3) and (4) of Example \ref{E:main} produce families which do not satisfy property (S).
For instance, take for $P_1$ and $P_2$ two non parallel segments with integer vertices in $\R^2$ and for $\calW_I$ the set of integer points in the relative interior of $P_I$.
If $v$ is a vertex of $P_1$, then $v+\calW_2$ is not contained in the relative interior of $P_1+P_2$, hence this family does not satisfy property (S).
A key lemma is the following one. It follows directly from Lemma \ref{L:keylemma}.

\begin{lemma}\label{L:partial}
Assume that $\bW=(\calW_I)_{I \subset [r]}$ satisfies property (S) with respect to $(\mu_1,\ldots,\mu_r)$.
%
For any nonempty $I \varsubsetneq [r]$ and any $\delta \in \R^n$ such that $\calW_I \cap (\delta+P_I )$ is not empty,
we have an injective map
$$\varphi_{\calS,I} : \calW_I \cap (\delta+P_I ) \longrightarrow
\bigsqcup_{\tiny Q \in \calS, \;\forall \, j \notin I, \; \dim Q_j=0} \calW \cap (\delta+Q)$$
mapping a point $w \in \delta+Q_I$, where $Q_I \in\calS_I$ has dimension $m$, to the point $w+v_{Q_I}$, where $v_{Q_I} \in \calS_{I^c}$ is the vertex such that
$Q_I+v_{Q_I}$ is an $m$-polytope $Q$ of $\calS$.
%
%
%
%
%
%
\end{lemma}

Later on we will speak about  $\varphi_{\calS,I}$ even if the source space $\calW_I \cap (\delta+P_I)$ is empty.
But in that case we make the convention that  the image of $\varphi_{\calS,I}$ is also empty.
We denote by $\mbox{Exc}_{\calS,I,\delta}$ the complementary part of the image of $\varphi_{\calS,I}$.
Set
$$
\mbox{Exc}_{\calS,\delta}=\bigcup_{\emptyset \neq I \varsubsetneq {\bf [}r {\bf ]}}\mbox{Exc}_{\calS,I,\delta},$$
and call {\it   excessive}  any point of $
\mbox{Exc}_{\calS,\delta}$.
\bigskip

\begin{figure}[htb]
\begin{center}
\includegraphics[scale=0.5]{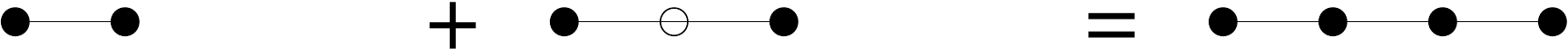}
\caption{An excessive point.}  
  \label{F:DiscExample6}\end{center}
\end{figure}

\begin{example}
\label{E:exc}
Figure \ref{F:DiscExample6} illustrates the simplest example of an excessive point. Here, $\calW_{1}=\{0,1\}$, $\calW_{2}=\{0,2\}$ and $\calW=\calW_{1}+\calW_{2}=\{0,1,2,3\}$. Consider the pure mixed subdivision of $[0,3]$ whose segments are $[0,1]+0$ and $1+[0,2]$.
Then $2 \in \calW \cap (1+[0,2])$ but $2 \notin 1+\calW_2$. Therefore, $2$ is an excessive point for this mixed subdivision and any small $\delta \in \R$.
\end{example}

%
%
%
%

Recall that the lower part $P_{\delta}$ of a polytope $P \subset \R^n$ with respect to $\delta \in \R^n$ is equal to $P$ if $\dim P <n$ or to
the union of all (closed) facets of $P$ with inward normal vectors $n$ such that
$\langle n , \delta \rangle >0$ if $\dim P=n$.
The following results are obvious.

\begin{lemma}\label{L:lower}
Let $P$ be a polytope in $\R^n$ and let $\calW$ be a subset of $P$. If $\delta \in \R^n$
is a sufficiently small vector,
then we have a disjoint union
$$
\calW =
\big(\calW \cap (\delta+P) \big)
\bigsqcup
\big(\calW \cap P_{\delta} \big)
.$$
 \end{lemma}

By small vector, we mean a vector with small norm.

\begin{lemma}\label{L:projection}
Let $P$ be a polytope in $\R^n$ and let $\calW$ be a subset of $P$. For any $\delta \in \R^n$,
the quotient map $\pi: \R^n \rightarrow \R^n / \langle \delta \rangle $ induces a bijection between
$\calW \cap P_{\delta}$ and $\pi(\calW \cap P_{\delta} )$.
%
\end{lemma}


\begin{lemma}
\label{L:stable}
Assume that $\bW=(\calW_I)_{I \subset [r]}$ satisfies property (S) with respect to $(\mu_1,\ldots,\mu_r)$.
Let $F=F_1+\cdots+F_r$ be a face of $P$, where $F_i$ is a face of $P_i$ for $i=1,\ldots,r$.
Then the family $\bW_F=(\calW_I \cap F_I)_{I \subset [r]}$, where $F_I=\sum_{i \in I} F_i$, satisfies property (S)
with respect to the restrictions of $\mu_1,\ldots,\mu_r$ to $F_1,\ldots, F_r$ respectively.
\end{lemma}

Given a family $(\calW_I)_{I \subset [r]}$, we consider the alternating sum
$$
N(\bW)=\sum_{I \subset [r]} {(-1)}^{r-|I|} |\calW_I|,$$
where the sum is over all subsets of $[r]$ including the emptyset, with the convention that
$|\calW_{\emptyset}|=1$.
For $\delta \in \R^n$, we also consider the quantity
$$
N_{\delta}(\bW)=\sum_{I \subset [r]} {(-1)}^{r-|I|} |\calW_I \cap (P_I)_{\delta}|,$$
with the convention that $|\calW_{\emptyset}\cap (P_{\emptyset})_{\delta}|=1$.

\begin{example}
\label{E:mainsum}
({\it Alternating sums $N(\bW)$ associated with families from Example \ref{E:main}.})
\medskip

\begin{itemize}
\item Consider the family in item (2). If $r=n$ then we get
$$N(\bW)=\sum_{I \subset [n]} {(-1)}^{r-|I|} |\Z^n \cap P_I|,$$
which as it is well-known coincides with the mixed volume $MV(P_1,\ldots,P_n)$.
\smallskip

\item The same conclusion holds true for the family in item (4) provided that $r=n$ and $\dim P_i=n$ for $i=1,\ldots,n$. 
\smallskip

\item Consider again the family in item (4). If $\dim P_i=n$ for $i=1,\ldots,n$ and $r \leq n$, then
$N(\bW)$ coincides with the genus of a generic complex complete intersection defined by polynomials with Newton polytopes
$P_1,\ldots,P_r$ in the projective toric variety associated with $P=P_1,\ldots,P_r$ (see~\cite{Kho1}).
\end{itemize}
\end{example}

Later on we will provide an interpretation of the sum $N(\bW)$ associated with the family in item (5) of Example \ref{E:main}
(see Proposition \ref{P:Nmixedcells}). Our central example is the alternating sum $N(\bW)$ associated with the family in item (1) of Example \ref{E:main}.
We call this sum {\it discrete mixed volume} of $\calW_1,\ldots,\calW_r)$ and denote it by
$D(\calW_1,\ldots,\calW_r)$. Thus,
$$
D(\calW_1,\ldots,\calW_r)=\sum_{I \subset [r]} {(-1)}^{r-|I|} |\sum_{i \in I} \calW_i|.$$
Section \ref{S:discrete} will be entirely devoted to this discrete mixed volume.

%
%

\begin{figure}[htb]
\begin{center}
\includegraphics[scale=0.5]{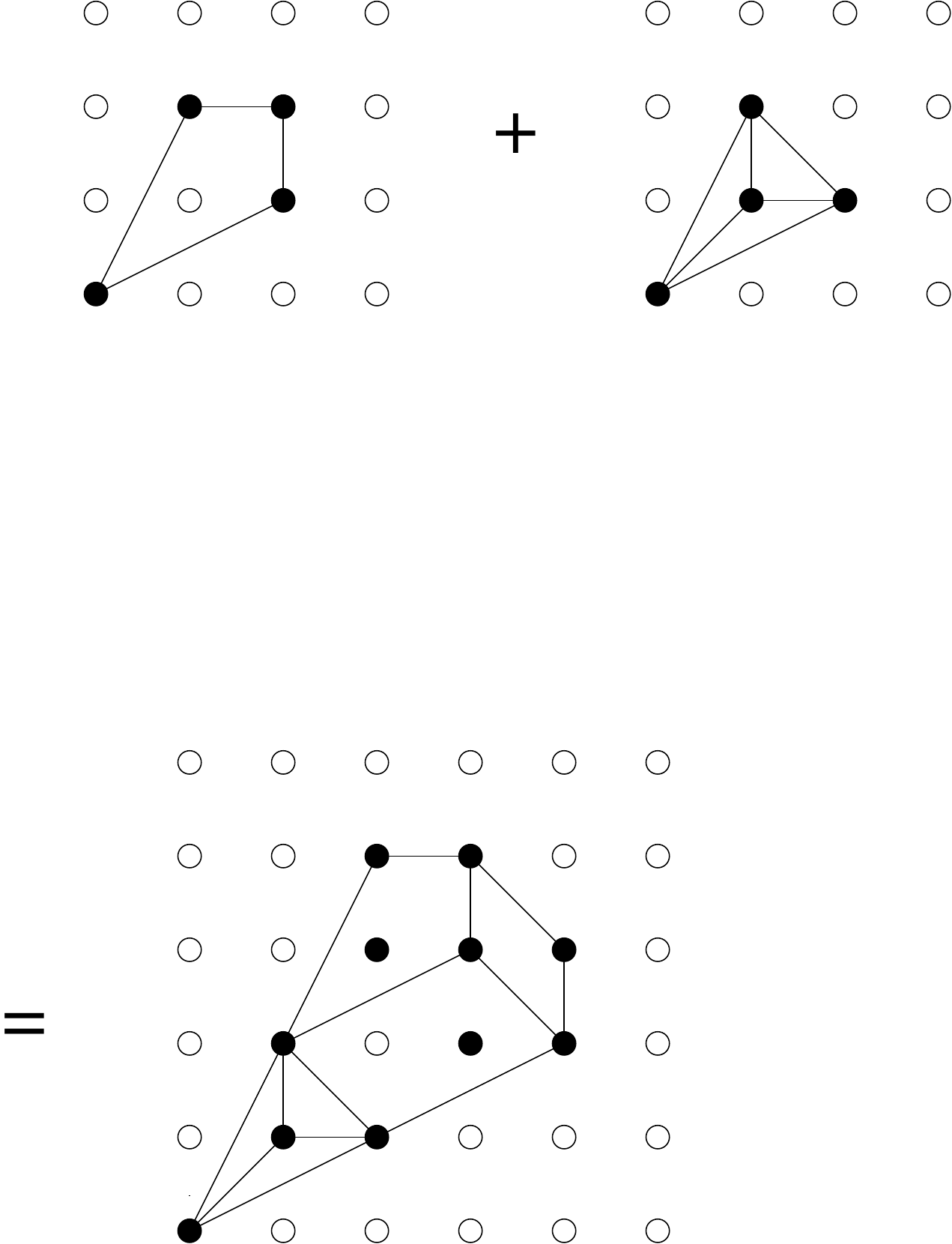}
\caption{ }  
\label{DiscExc1}\end{center}
\end{figure}

\begin{prop}\label{P:Nsum}

Assume that $\bW=(\calW_I)_{I \subset [r]}$ satisfies property (S) with respect to $(\mu_1,\ldots,\mu_r)$.
If $\delta$ is a sufficiently small vector of $\R^n$ which is not parallel to any polytope of $\calS$, then
\begin{equation}\label{E:Nsumbis}
N(\bW)=A+|\mbox{Exc}_{\calS,\delta}|+N_{\delta}(\bW)
\end{equation}
where $A$ is the sum $\sum_{Q \in \calS, \, \dim(Q_i) \geq 1} |\calW \cap (\delta+Q)|$
over all polytopes $Q=Q_1+\cdots+Q_r \in \calS$ such that  $\dim(Q_i) \geq 1$ for $i=1,\ldots,r$ (in other words, over all mixed polytopes $Q \in \calS$).
\end{prop}
\begin{proof}
Since $\delta$ is not parallel to a polytope of $\calS$, the set $\calW \cap (\delta+P)$ is the disjoint union of the sets
$\calW \cap (\delta+Q)$ over all $n$-dimensional polytopes $Q$ of $\calS$.
This gives using Lemma \ref{L:partial}
$$
|\calW \cap (\delta+P)|=A+|\mbox{Exc}_{\calS,\delta}|+|\bigcup_{\emptyset \neq I \varsubsetneq {\bf [}r {\bf ]}}\mbox{Im} \, \varphi_{\calS,I}|
.$$
By inclusion-exclusion principle and injectivity of $\varphi_{\calS,I}$, we get
$$
|\calW  \cap (\delta+P)|=A+|\mbox{Exc}_{\calS,\delta}|-
\sum_{\emptyset \neq I \varsubsetneq {\bf [}r {\bf ]}} (-1)^{r-|I|} 
\big|\calW_I \cap (\delta+P_I) \big|,
$$
and thus
$$
\sum_{\emptyset \neq I \subset {\bf [}r {\bf ]}} (-1)^{r-|I|} 
\big|\calW_I \cap (\delta+P_I) \big|=A+|\mbox{Exc}_{\calS,\delta}|.
$$

For $\delta$ small enough, Lemma \ref{L:lower} gives then

$$
\sum_{\emptyset \neq I \subset {\bf [}r {\bf ]}} (-1)^{r-|I|} 
\big|\calW_I \big|
-
\sum_{\emptyset \neq I \subset {\bf [}r {\bf ]}} (-1)^{r-|I|} 
\big| \calW_I  \cap \big(P_I\big)_{\delta} \big|=
A+|\mbox{Exc}_{\calS,\delta}|
,$$
and the result follows.
\end{proof}

\begin{cor}
Assume that $\bW=(\calW_I)_{I \subset [r]}$ satisfies property (S).
Then for all sufficiently small generic vector $\delta \in \R^n$, we have $N(\bW) \geq N_{\delta}(\bW)$.
\end{cor}

\begin{ex}\label{E:excplusdefec}
Let $\calW_1=\{(0,0),(2,1),(1,2), (2,2)\}$ and $\calW_2=\{(0,0),(2,1),(1,2),(1,1)\}$.
We consider here the discrete mixed volume associated with $\calW_1$ and $\calW_2$.
In Figure \ref{DiscExc1} we have depicted a pure convex mixed subdivision $\calS$ of $P_1+P_2$ associated with some functions $\mu_i:\calW_i \rightarrow \R$
(where $P_i$ is the convex-hull of $\calW_i$). The sets $\calW_1$, $\calW_2$ and $\calW=\calW_1+\calW_2$ are in bold in Figure \ref{DiscExc1}.
The mixed subdivision $\calS$ contains six two-dimensional polytopes, two of them are mixed polytopes,
one non-mixed polytope is $P_1+(1,2)$, the three others non-mixed polytopes are triangles which subdivide $(0,0) +P_2$. The point $(2,3)$ is an excessive point for this subdivision
(with respect to any small generic vector $\delta$) since
$(2,3)=(1,2)+(1,1) \in \calW=\calW_1+\calW_2$, $(2,3) \in P_1+(1,2) \in \calS$ while $(2,3)-(1,2)=(1,1) \notin \calW_1$. Here, we find that
$D(\calW_1,\calW_2)=|\calW|-|\calW_1|-\calW_2|+1=4$.

Note also that $MV(P_1,P_2)=4$, which is the total area of the two mixed polytopes. If we interpret $MV(P_1,P_2)$ as the sum $N(\bW)$ associated with the family defined by
$\calW_I=\Z^2 \cap P_I$ (item (2) of Example \ref{E:main}), we see that $(2,3)$ is not an excessive point for this family and the considered mixed subdivision,
but on the other hand $(2,2) \in \Z^2 \setminus \calW$.
Both points compensate each other when computing separatively
$N(\bW)$ and $D(\calW_1,\calW_2)$ using Proposition \ref{P:Nsum}.
\end{ex}

\begin{figure}[htb]
\begin{center}
\includegraphics[scale=0.3]{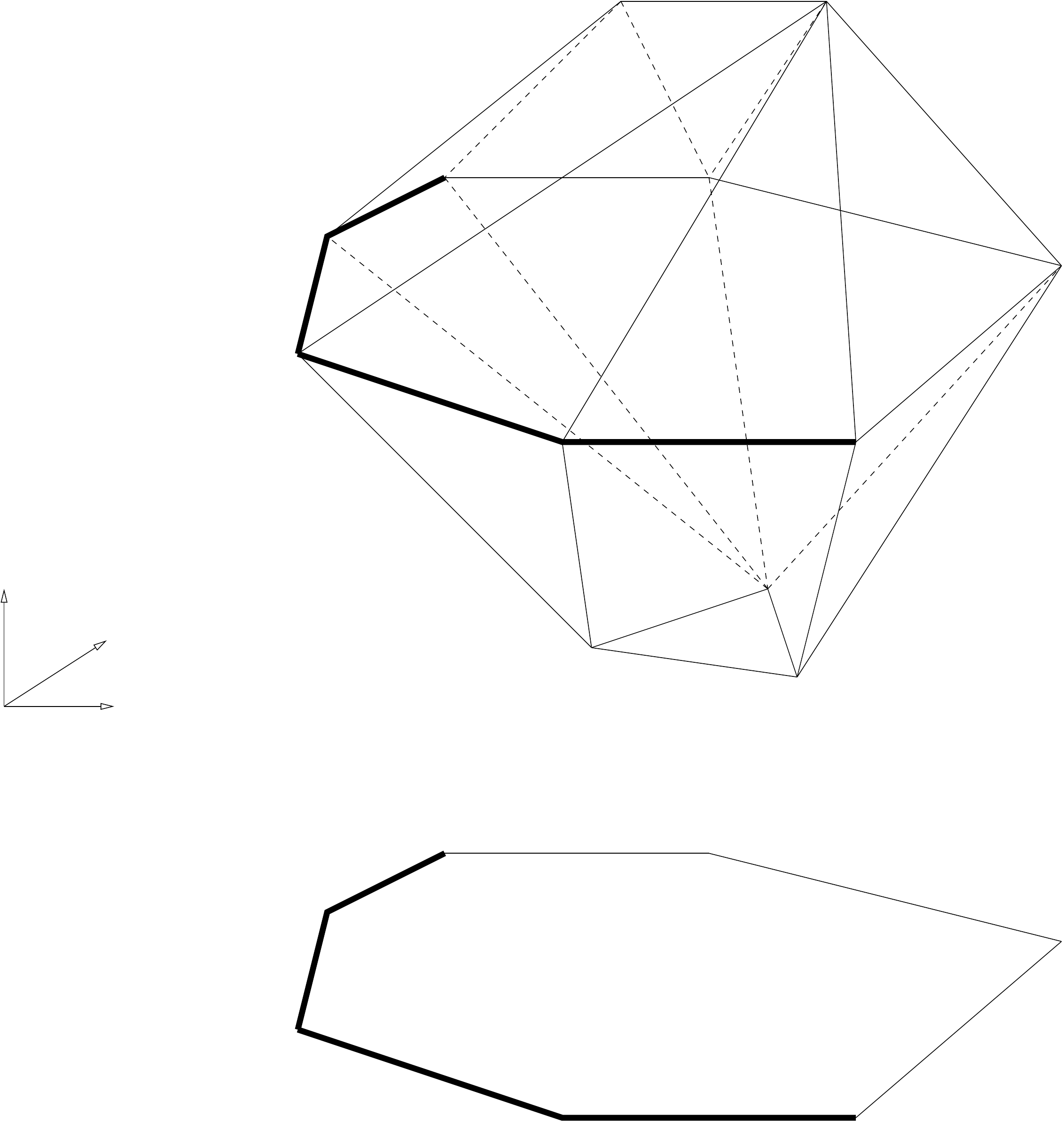}
\caption{A piecewise linear manifold $P_{\delta_1,\delta_2} \subset P \subset \R^3$
and the polygon $\Delta_1 \subset \R^2$ (here $\delta_1=(0,0,1)$ and $\delta_2=(1,0,0)$). }  
  \label{F:piecewise}\end{center}
\end{figure}
\bigskip

Let $P \subset \R^n$ be a polytope and let $\delta_1,\ldots,\delta_n$ be linearly independent vectors in $\R^n$.
Let $k \in [n]$.  Denote by $\pi_k$ the restriction to $P$ of the quotient map $\R^n \rightarrow \R^n / \langle \delta_1,\ldots, \delta_k \rangle$.
We define inductively a piecewise-linear manifold $P_{\delta_1,\ldots,\delta_k} \subset P$ and denote by $\Delta_k$ its image under the quotient map $\pi_k$.
We also check inductively that $\Delta_k$ is a polytope in $\R^n / \langle \delta_1,\ldots, \delta_k \rangle$.
If $k=1$, then $P_{\delta_1}$ is the lower part of $P$ with respect to $\delta_1$ and obviously $\Delta_1$ is a polytope in $\R^n / \langle \delta_1 \rangle$.
Let $k \in [n-1]$. Assume that $P_{\delta_1,\ldots,\delta_{k}}$ has already been defined and that $\Delta_k$  is a polytope.
Then take the lower part of $\Delta_{k}$ with respect to $\pi_{k}(\delta_{k+1})$ and define $P_{\delta_1,\ldots,\delta_{k+1}}$ as its inverse image under $\pi_{k}$:
 $$P_{\delta_1,\ldots,\delta_{k+1}}= \pi_{k}^{-1} \left( {(\Delta_{k})}_{\small{\pi_{k}(\delta_{k+1})}} \right).$$
It is obvious that $\Delta_{k+1}$ is a polytope. Note that $P_{\delta_1,\ldots,\delta_{k+1}} \subset P_{\delta_1,\ldots,\delta_{k}}$ and that
$P_{\delta_1,\ldots,\delta_k}$ is a contractible connected piecewise-linear manifold contained in $P$ whose linear pieces are faces of $P$, see Figure \ref{F:piecewise}.

Recall that $P$ is polytope in $\R^n$ and that $m=\dim P$.
\begin{lemma}\label{L:dim}
We have $P_{\delta_1,\ldots,\delta_k}=P$ for $k=1,\ldots,n-m$ and $\dim  P_{\delta_1,\ldots,\delta_k} = n-k$ for $k \geq n-m$.
In particular, $P_{\delta_1,\ldots,\delta_n}$ is a vertex of $P$.
\end{lemma}
\begin{proof}
This follows by induction using the facts that $P_{\delta}=P$ when $P \subset \R^n$ has dimension smaller than $n$ and $\dim P_{\delta} = n-1$ when $\dim P=n$.
\end{proof}

One may observe that if $\dim P=n$ then
$$P_{\delta_1,\ldots,\delta_{k+1}} =  \left(\bigcap_{i=1}^k (P_{\delta_i} \cap P_{-\delta_i})\right)  \cap P_{\delta_{k+1}},$$
where $k \in [n-1]$. Note that the condition $\dim P=n$ cannot be removed since if $\dim P <n$ then the right member of the previous equality is equal to $P$.

For a given family $\bW=(\calW_I)_{I \subset [r]}$ and  linearly independent vectors $\delta_1,\ldots,\delta_n$ in $\R^n$, we define
$$
N_{\delta_1,\ldots,\delta_k}(\bW)=\sum_{I \subset [r]} {(-1)}^{r-|I|} |\calW_I \cap (P_I)_{\delta_1,\ldots,\delta_k}|,$$
where $k \in [n]$ and with the convention that $|\calW_{\emptyset}\cap (P_{\emptyset})_{\delta_1,\ldots,\delta_k}|=1$.
We now relate $N_{\delta_1,\ldots,\delta_k}(\bW)$ with $N_{\delta_1,\ldots,\delta_{k+1}}(\bW)$ assuming that $\bW$ satisfies property (S).
Let $k \in [n-1]$ and assume that $\bW$ satisfies property (S) with respect to functions $\mu_1,\ldots,\mu_r$. By Lemma \ref{L:dim},
we have $\dim P_{\delta_1,\ldots,\delta_k} \leq n-k$ and $P_{\delta_1,\ldots,\delta_k}=P_{\delta_1,\ldots,\delta_{k+1}}$ when this inequality is strict.
Therefore, if $\dim P_{\delta_1,\ldots,\delta_k} < n-k$ then $N_{\delta_1,\ldots,\delta_k}(\bW)=N_{\delta_1,\ldots,\delta_{k+1}}(\bW)$.

Assume that $P_{\delta_1,\ldots,\delta_k}$ has dimension $n-k$ and denote by $F_1,\ldots,F_s$ its facets.
Write each facet $F_j$ as $F_j=\sum_{i=1}^r F_{i,j}$, where $F_{i,j}$ is a face of $P_i$.
Let $\calS$ be the pure convex mixed subdivision of $P$ associated with $\mu_1,\ldots,\mu_r$.
Denote by $\calS_j$ the pure convex mixed subdivision of $F_j$ associated with the restrictions of $\mu_1,\ldots,\mu_r$
to $F_{1,j},\ldots,F_{r,j}$ respectively. We simply have $\calS_j=\calS \cap F_j$.
By Lemma \ref{L:stable}, the family $\bW_{F_j}=(\calW_I \cap F_{I,j})_{I \subset [r]}$, where $F_{I,j}=\sum_{i \in I} F_{i,j}$, satisfies property (S)
with respect to the restrictions of $\mu_1,\ldots,\mu_r$ to $F_{1,j},\ldots,F_{r,j}$ respectively. In order to apply Lemma \ref{L:partial}, we need to project this family
onto $\R^n/\langle \delta_1,\ldots,\delta_k \rangle$.

Set $F'_{I,j}=\pi_k(F_{I,j})$, $F'_j=\pi_k(F_j)$, $F'_{i,j}=\pi_k(F_{i,j})$,
$\calW'_{I,j}=\pi_k(\calW_I \cap F_{I,j})$, $\calW'_{j}=\pi_k(\calW \cap F_{j})$ and $\calW'_{i,j}=\pi_k(\calW_i \cap F_{i,j})$.
Define $\mu'_{i,j}:\pi_{k}(\calV_i \cap F_{i,j}) \rightarrow \R$, $x \mapsto (\mu_i \circ \pi_k^{-1})(x)$. Then $\calS_{i,j}'=\pi_k(\calS_{i,j})$ is the convex subdivision of $F_{i,j}'$ associated with $\mu'_{i,j}$ and $\calS_j'=\pi_k(\calS_j)$ is the pure convex mixed subdivision of $F'_j$ associated with $\mu'_{1,j},\ldots,\mu'_{r,j}$.
Moreover,  projecting $\bW_{F_j}$ onto $\R^n/\langle \delta_1,\ldots,\delta_k \rangle$, we get the family
$\bW_{F'_j}=(\calW'_{I,j})_{I \subset [r]}$ which satisfies property (S) with respect to $\mu'_{1,j},\ldots,\mu'_{r,j}$.
Note that $F'_j$ has full dimension in $\R^n/\langle \delta_1,\ldots,\delta_k \rangle$ since we assumed that $P_{\delta_1,\ldots,\delta_k}$ has dimension $n-k$
and $\delta_1,\ldots,\delta_k$ are linearly independent.

We apply Lemma \ref{L:partial} to the family $\bW_{F'_j}$.
This gives for any nonempty proper subset $I \subset [r]$ and $\delta'=\pi_k(\delta_{k+1})$
an injective map
$$\varphi_{\calS'_j,I}: \calW'_{I,j} \bigcap (\delta'+F'_{I,j} ) \longrightarrow
\bigsqcup_{\tiny Q \in \calS'_j, \;\forall \, \ell \notin I, \; \dim Q_{\ell}=0} \calW'_{j} \cap (\delta+Q),$$
where $Q_{\ell} \in \calS'_{\ell,j}$ is a summand of $Q$ written as a sum of polytopes of $\calS'_{1,j},\ldots,\calS'_{r,j}$ respectively.
Denote by $A'_j$  the sum $\sum_{Q \in \calS'_j, \, \dim(Q_{\ell}) \geq 1} |\calW'_{j}\cap (\delta'+Q)|$
over all polytopes $Q=Q_1+\cdots+Q_r \in \calS'_j$ such that  $\dim(Q_{\ell}) \geq 1$ for $\ell=1,\ldots,r$.
Denote by $\mbox{Exc}_{\calS'_j,I,\delta'}$ the complementary part of the image of $\varphi_{\calS'_j,I}$
and set $\mbox{Exc}_{\calS'_j,\delta'}=\bigcup_{I \varsubsetneq {\bf [}r {\bf ]}}\mbox{Exc}_{\calS'_j,I,\delta'}$.

\begin{prop}\label{P:Nsumg}
Assume that $\bW=(\calW_I)_{I \subset [r]}$ satisfies property (S) with respect to $(\mu_1,\ldots,\mu_r)$.
If $\dim P_{\delta_1,\ldots,\delta_k} <n-k$ then $N_{\delta_1,\ldots,\delta_{k}}(\bW)=N_{\delta_1,\ldots,\delta_{k+1}}(\bW)$.
If $\dim P_{\delta_1,\ldots,\delta_k} =n-k$ and if $\delta'$ is a sufficiently small vector not parallel to a polytope of $\calS_1' \cup \cdots \cup \calS_s'$,
then 
$$
N_{\delta_1,\ldots,\delta_{k}}(\bW)=N_{\delta_1,\ldots,\delta_{k+1}}(\bW)+\sum_{j=1}^s \left(A'_j+|\mbox{Exc}_{\calS'_j,\delta'}|\right),$$
\end{prop}
\begin{proof}
We already showed that if $\dim P_{\delta_1,\ldots,\delta_k} <n-k$ then $N_{\delta_1,\ldots,\delta_{k}}(\bW)=N_{\delta_1,\ldots,\delta_{k+1}}(\bW)$.
So assume that $\dim P_{\delta_1,\ldots,\delta_k} =n-k$.
Set $P'=\pi_{k}(P_{\delta_1,\ldots,\delta_k})$ and $P'_I=\pi_{k}({(P_I)}_{\delta_1,\ldots,\delta_k})$ for any $I \subset [r]$.
Since $\delta'$ is not parallel to a polytope of $\calS_1' \cup \cdots \cup \calS_s'$,
we have a disjoint union
$$
\calW' \cap (\delta'+P') = \bigcup_{j=1}^s \left(\calW'_j \cap (\delta'+F'_j) \right).$$
Moreover, arguing as in the proof of Proposition \ref{P:Nsum}, we obtain
$$
|\calW'_j \cap (\delta'+F_j')|=A'_j+|\mbox{Exc}_{\calS_j',\delta'}|-
\sum_{\emptyset \neq I \varsubsetneq {\bf [}r {\bf ]}} (-1)^{r-|I|} 
\big|\calW'_I \cap (\delta'+F'_{I,j}) \big|.$$
Taking the sum for $j=1,\ldots,s$ yields
$$
|\calW' \cap (\delta'+P')|=\sum_{j=1}^s \left(A'_j+|\mbox{Exc}_{\calS'_j,\delta'}|\right)
-\sum_{\emptyset \neq I \varsubsetneq {\bf [}r {\bf ]}} (-1)^{r-|I|} 
\sum_{j=1}^s \big|\calW'_I \cap (\delta'+F'_{I,j}) \big|.$$
But since $\sum_{j=1}^s \big|\calW'_I \cap (\delta'+F'_{I,j}) \big|=\big|\calW'_I \cap (\delta'+P'_I) \big|$, we get
$$
\sum_{\emptyset \neq I \subset {\bf [}r {\bf ]}} (-1)^{r-|I|} \big|\calW'_I \cap (\delta'+P'_I) \big|=\sum_{j=1}^s \left(A'_j+|\mbox{Exc}_{\calS'_j,\delta'}|\right)
.$$
For $\delta'$ small enough Lemma \ref{L:lower} yields
$\big|\calW'_I \cap (\delta'+P'_I) \big|=\big|\calW'_I \cap (P'_I) \big|-\big|\calW'_I \cap (P'_I)_{\delta'} \big|$,
and It remains to use Lemma \ref{L:projection}.
\end{proof}

\begin{cor} \label{C:lowerdiscrete}
If $\bW=(\calW_I)_{I \subset [r]}$ satisfies property (S), then for any sufficiently small generic (and thus linearly independent) vectors
$\delta_1,\ldots,\delta_n$ in $\R^n$, we have
$$N(\bW) \geq N_{\delta_1}(\bW) \geq  N_{\delta_1,\delta_2}(\bW) \geq \cdots \geq N_{\delta_1,\ldots,\delta_n}(\bW).$$
\end{cor}

Assume that $\bW=(\calW_I)_{I \subset [r]}$ satisfies property (S) and let $\delta_1,\ldots,\delta_n \in \R^n$ be sufficiently small generic linearly independent vectors.
For any nonempty $I \subset [r]$, we get a contractible piecewise-linear submanifold $(P_I)_{\delta_1,\ldots,\delta_n}$ of $P_I$ which has dimension zero (see Lemma \ref{L:dim}).
Thus $(P_I)_{\delta_1,\ldots,\delta_n}$ is a vertex of $P_I$ for any nonempty $I \subset [r]$. In fact, if $u_i$ is the vertex $(P_i)_{\delta_1,\ldots,\delta_n}$, then $(P_I)_{\delta_1,\ldots,\delta_n}$ is the point $\sum_{i \in I} u_i$, that we denote by $u_I$. The definition of $N_{\delta_1,\ldots,\delta_{n}}(\bW)$ translates then into the following result.

\begin{prop}\label{P:zero}
If $\bW=(\calW_I)_{I \subset [r]}$ satisfies property (S), then for any sufficiently small generic (and thus linearly independent) vectors
$\delta_1,\ldots,\delta_n$ in $\R^n$, we have
$$N_{\delta_1,\ldots,\delta_{n}}(\bW)=(-1)^r
+\sum_{\emptyset \neq I \subset [r]} (-1)^{r-|I|} \ind_{u_I \in \calW_I},$$
where $\ind_{u_I \in \calW_I}=1$ if $u_I \in \calW_I$ and $0$ otherwise.
\end{prop}
%

\begin{thm}\label{T:nonnegative}
We have the following nonnegativity properties.
\begin{enumerate}
\item 
If $\calW_1,\ldots,\calW_r$ are finite nonempty subsets of $\R^n$ then
$$\sum_{I \subset {\bf [}r {\bf ]}} (-1)^{r-|I|} 
\big|\sum_{i \in I} \calW_i \big| \geq 0.$$

\item 
If $P_1,\ldots,P_r$ are lattice polytopes in $\R^n$ then
$$\sum_{I \subset {\bf [}r {\bf ]}} (-1)^{r-|I|} 
\big|\Z^n \cap P_I\big| \geq 0.$$

\item If $\calW_1,\ldots,\calW_r$ are finite nonempty subsets of $\R^n$ with convex-hulls $P_1,\ldots,P_r$ respectively,
then
$$\sum_{\emptyset \neq I \subset {\bf [}r {\bf ]}} (-1)^{r-|I|} 
\big|(\sum_{i \in I} \calW_i) \cap Int_{abs}(P_I) \big| \geq 0,$$
where $Int_{abs}(P_I)$ is the absolute interior of $P_I$.

\item Let $P_1,\ldots,P_r$ be lattice polytopes in $\R^n$. With the same notations, we have
$$\sum_{\emptyset \neq I \subset {\bf [}r {\bf ]}} (-1)^{r-|I|} 
\big| \Z^n \cap Int_{abs}(P_I) \big| \geq 0.$$

\item Let $P_1,\ldots,P_r$ be polytopes in $\R^n$ and $\calS$ be a pure convex mixed subdivision of $P=P_1+\cdots+P_r$ associated with functions
$\mu_1,\ldots,\mu_r$. If $\calV_I$ is the set of vertices of the convex subdivision of $P_I$ associated with $\mu_i$ for $i \in I$, then
$$\sum_{ I \subset {\bf [}r {\bf ]}} (-1)^{r-|I|} 
\big| \calV_I \big| \geq 0.$$
\end{enumerate}
\end{thm}
\begin{proof}
We apply Corollary \ref{C:lowerdiscrete} and Proposition \ref{P:zero}.
In items (1), (2) and (5), we have $u_I \in\calW_I$ for any nonempty $I \subset [r]$ and thus $N(\bW) \geq N_{\delta_1,\ldots,\delta_{n}}(\bW)=(-1)^r
+\sum_{\emptyset \neq I \subset [r]} (-1)^{r-|I|}=(1-1)^r=0$. In items (3) and (4), we have $u_I \notin\calW_I$ for any nonempty $I \subset [r]$ and thus
$N(\bW) -(-1)^r \geq N_{\delta_1,\ldots,\delta_{n}}(\bW)-(-1)^r=\sum_{\emptyset \neq I \subset [r]} 0=0$.
\end{proof}

To our knowledge, the inequalities of Theorem \ref{T:nonnegative} are new with a small number of exceptions listed in Example \ref{E:mainsum},
where $N(\bW)$ is known to coincide with a mixed volume or with the genus of complex variety.

Soprunov's Problem 2 in~\cite{BNRSS} asks for a combinatorial proof of the inequality $|\Z^n \cap Int_{abs}(P)| \geq MV(P_1,\ldots,P_n)-1$ which holds true
for any $n$-dimensional lattice polytopes $P_1,\ldots,P_n \subset \R^n$. As explained in Nill's note~\cite{N}, our proof of item (4) in Theorem \ref{T:nonnegative}
provides an answer to Soprunov's Problem. We also note that another combinatorial proof of the case $r=n-1$ of item (2) has been obtained in~\cite{ST}.

Our mixed irrational decomposition trick allows us to obtain the following interpretation
of the sum $N(\bW)$ associated with the family (5) in Example \ref{E:main}.

\begin{prop}\label{P:Nmixedcells}
Let $P_1,\ldots,P_n$ be polytopes in $\R^n$ and $\calS$ be a pure convex mixed subdivision of $P=P_1+\cdots+P_n$ associated with functions
$\mu_1,\ldots,\mu_n$. If $\calV_I$ is the set of vertices of the convex subdivision of $P_I$ associated with $\mu_i$ for $i \in I$, then
$$\sum_{I \subset {\bf [}n {\bf ]}} (-1)^{n-|I|} 
\big| \calV_I \big|$$
coincides with the number of mixed polytopes of
$\calS$. 
\end{prop}
\begin{proof}
Denote by $\bW$ the family $(\calV_I)_{I \subset [n]}$,  which as we know satisfies property (S).
Let $\delta$ be a sufficiently small generic vector in $\R^n$.
Proposition \ref{P:Nsum} gives $N(\bW)=A+|\mbox{Exc}_{\calS,\delta}|+N_{\delta}(\bW)$.
But in our case $A$ is exactly the number mixed polytopes and obviously $\mbox{Exc}_{\calS,\delta}$ is empty.
Proposition \ref{P:Nsumg} yields $N_{\delta_1,\ldots,\delta_{k}}(\bW)=N_{\delta_1,\ldots,\delta_{k+1}}(\bW)$
since for this family we have $|\mbox{Exc}_{\calS'_j,\delta'}|=0$ and $A'_j=0$ for $j=1,\ldots,s$. Thus $N_{\delta}(\bW)=N_{\delta_1, \ldots,\delta_n}(\bW)$ for sufficiently small generic
linearly independent vectors $\delta_1,\ldots,\delta_n \in \R^n$, and the result follows as
$N_{\delta_1,\ldots,\delta_n}(\bW)=(1-1)^n=0$.
\end{proof}

We note that an almost identical proof allows us to show the classical equality $N(\bW)=
MV(P_1,\ldots,P_n)$ for the family $\bW$ defined by $\calW_I=\Z^n \cap P_I$,
see Example \ref{E:mainsum}.


\section{Discrete mixed volume and fewnomial bounds for tropical systems}
\label{S:discrete}

Let $\calW_1,\ldots,\calW_{r}$ be finite subsets of $\R^n$ and consider the associated discrete mixed volume
\begin{equation}\label{E:N}
D(\calW_1,\ldots,\calW_{r})=\sum_{I \subset {\bf [}r {\bf ]} } {(-1)}^{r-|I|}
\big| \sum_{i \in I} \calW_i \big|.
\end{equation}
Later on we will simply write $\calW_I$ for $\sum_{i \in I} \calW_i$.
One simple observation is the following one.

\begin{prop}\label{P:bijective}
Assume that for any nonempty $I\subset {\bf [}r {\bf ]}$,  the map $\prod_{i \in I} \calW_i \mapsto \sum_{i \in I} \calW_i$, $(w_i)_{i \in I} \mapsto \sum_{i \in I} w_i$ is bijective.
Then
$$D(\calW_1,\ldots,\calW_{r})=\prod_{i \in {\bf [}r {\bf ]}} (|\calW_i|-1).$$
\end{prop}
\begin{proof}
Indeed, we have $\prod_{i \in {\bf [}r {\bf ]}} (|\calW_i|-1)=\sum_{I \subset {\bf [}r {\bf ]} } {(-1)}^{r-|I|}
\prod_{i \in I} |\calW_i |$.
\end{proof}

\begin{cor}
\label{C:general}
If the dimensions of the affine spaces spanned by $\calW_1,\ldots,\calW_r$ sum up to the dimension of the affine spaced spanned by $\calW_1+\cdots+\calW_r$, or if
$\calW_1,\ldots,\calW_r$ are in general position, then
$D(\calW_1,\ldots,\calW_{r})=\prod_{i \in {\bf [}r {\bf ]}} (|\calW_i|-1)$.
\end{cor}

\begin{figure}[htb]
\begin{center}
\includegraphics[scale=0.4]{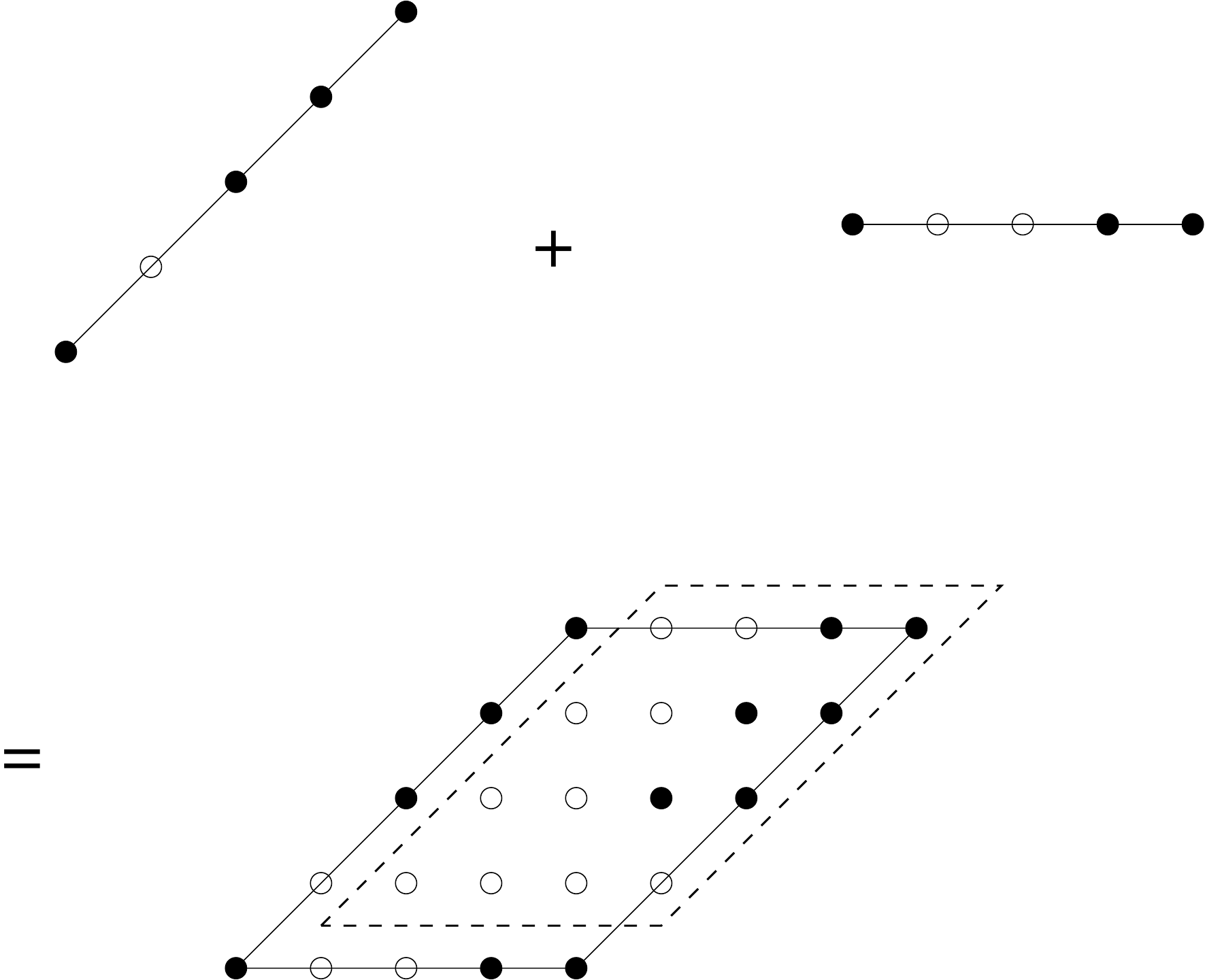}
\caption{Discrete mixed volume of a mixed polytope.}  
  \label{F:discretemixed polytope}\end{center}
\end{figure}
\bigskip

The next result has a to be compared with Proposition \ref{P:brick}. It shows some similarity between the classical mixed volume and the discrete one. 

\begin{prop}\label{P:brickbis}
Let $\calW_1,\ldots,\calW_{n}$ be finite subsets of $\R^n$ with convex-hulls $Q_1,\ldots,Q_n$, respectively. Assume that $Q=Q_1+\cdots+Q_n$ has dimension $n$ and each $Q_i$ is a segment. Then for any sufficiently small vector $\delta \in \R^n$ not parallel to a face of  $Q$, we have
\begin{equation}\label{E:Nmixed polytope}
D(\calW_1,\ldots,\calW_{n})=|(\calW_1 +  \cdots + \calW_{n}) \cap (\delta+Q)|
\end{equation}
\end{prop}
\begin{proof}
By Corollary \ref{C:general} we have $D(\calW_1,\ldots,\calW_{n})=\prod_{i=1}^n (|\calW_i|-1)$.
It it is not difficult to show using Lemma \ref{L:lower} that $\prod_{i=1}^n (|\calW_i|-1)=|(\calW_1 +  \cdots + \calW_{n}) \cap (\delta+Q)|$
when $Q=Q_1+\cdots+Q_n \subset \R^n$ has dimension $n$ and each $Q_i$ is a segment (see Figure \ref{F:discretemixed polytope}). \end{proof}

\begin{figure}[ht]
\begin{center}
\includegraphics[scale=0.4]{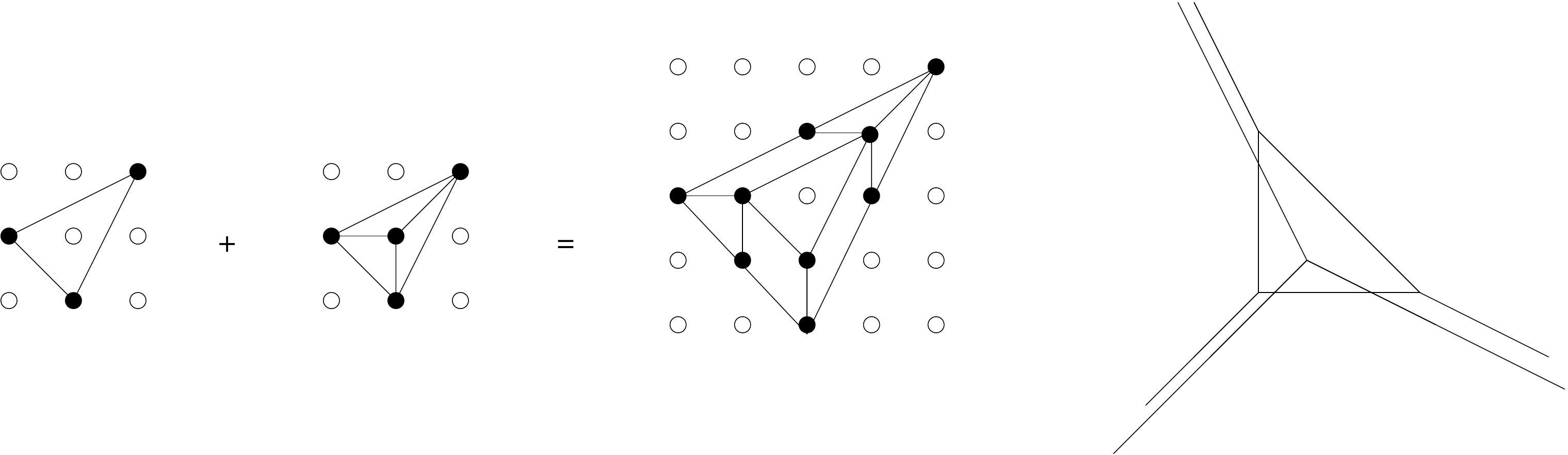}
\caption{ Three intersection points and $D(\calW_1,\calW_2)=3$.}  
\label{DiscExamplebig}\end{center}
\end{figure}

\begin{figure}[ht]
\begin{center}
\includegraphics[scale=0.4]{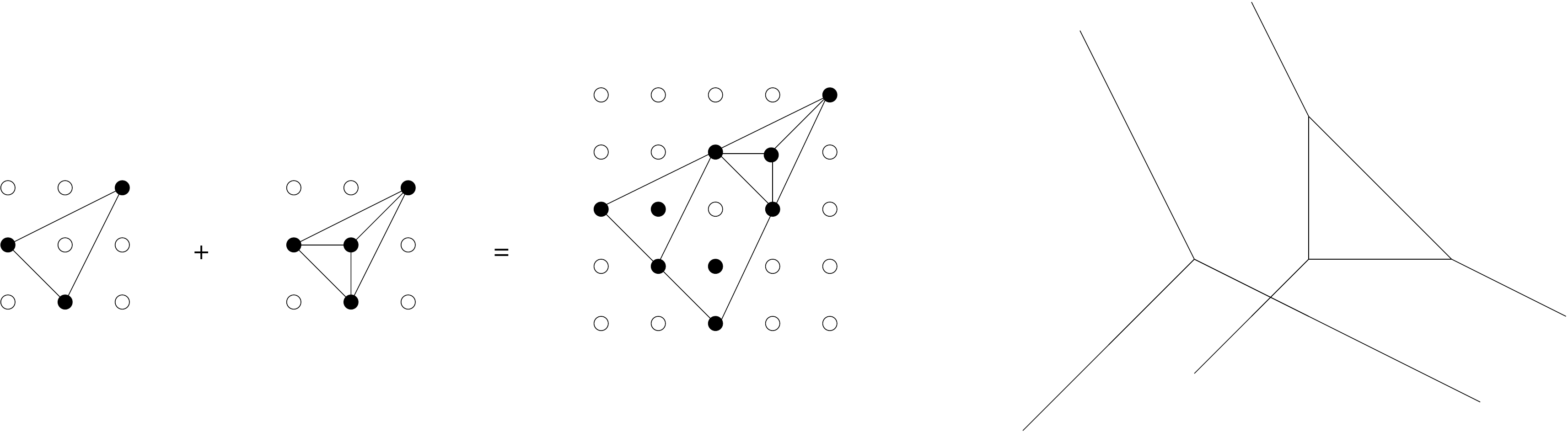}
\caption{One intersection point and $D(\calW_1,\calW_2)=3$.}  
\label{DiscExamplebigbis}\end{center}
\end{figure}

The main result of this section is the following one.

\begin{thm}
\label{T:Mainbis}
The number of nondegenerate solutions of a tropical polynomial system
with supports $\calW_1,\ldots,\calW_n \subset \R^n$
is bounded from above by $D(\calW_1,\ldots,\calW_{n})$.
\end{thm}
\begin{proof}

By translating the tropical hypersurfaces by small vectors, we obtain a tropical polynomial
system with at least the same number of nondegenerate solutions and whose dual mixed subdivision $\calS$ is pure. Let $\delta$ be any sufficiently small generic vector in $\R^n$.
Denote by $\bW$ the family $(\calW_I)_{i \in I}$. We apply Proposition \ref{P:Nsum} to obtain
$D(\calW_1,\ldots,\calW_{n})=A+|\mbox{Exc}_{\calS,\delta}|+N_{\delta}(\bW)$. It is easy to show that if $Q=Q_1+\cdots+Q_n$ is a $n$-dimensional polytope with
$\dim Q_i=1$ for $i=1,\ldots,n$, and if $\delta$ is a sufficiently small vector in $\R^n$, then $\delta+Q$ contains at least one vertex of $Q$.
Since the vertices of $\calS$ belong to $\calW$, it follows that $|\calW \cap (\delta+Q)| \geq 1$ for any mixed polytope $Q \in \calS$. Thus
$A$ is bounded from below by the number of mixed polytopes $Q \in \calS$. By Corollary \ref{C:lowerdiscrete}, we have $N_{\delta}(\bW) \geq N_{\delta, \delta_2, \ldots,\delta_n}(\bW)$ for
any sufficiently small generic linearly independent vectors $\delta,\delta_2,\ldots,\delta_n \in \R^n$. Applying Proposition \ref{P:zero} yields then $N_{\delta, \delta_2, \ldots,\delta_n}(\bW)=0$.
Therefore, the number of mixed polytopes of $\calS$ does not exceed $D(\calW_1,\ldots,\calW_{n})$.
\end{proof}

Note that in Theorem \ref{T:Mainbis} we do not assume that $\calW_1,\ldots,\calW_n \subset \Z^n$. Consider finite sets $\calW_1,\ldots,\calW_n \subset \Z^n$
with convex-hulls $P_1,\ldots,P_n$ respectively. We want to compare $D(\calW_1,\ldots,\calW_{n})$ with $MV(P_1,\ldots,P_n)$. We again apply Proposition \ref{P:Nsum} to obtain
$$D(\calW_1,\ldots,\calW_{n})=A+|\mbox{Exc}_{\calS,\delta}|+N_{\delta}(\bW).$$

Recall that $A$ is the sum over all mixed polytopes $Q \in \calS$ of the quantities $|\calW \cap (\delta+Q)|$.
On the other hand, the mixed volume $MV(P_1,\ldots,P_n)$ is the sum over the same polytopes of the quantities $|\Z^n\cap (\delta+Q)|$
(see Proposition \ref{P:brick} and Proposition \ref{P:mixedvolumesum}). Thus $A \leq MV(P_1,\ldots,P_n)$, and
if $|\mbox{Exc}_{\calS,\delta}|=N_{\delta}(\bW)=0$, then $D(\calW_1,\ldots,\calW_{n}) \leq MV(P_1,\ldots,P_n)$.
However, the latter inequality is not always true due to the presence of excessive points for $\calS$, or for intermediate subdivisions $\calS_j'$ as considered in Proposition \ref{P:Nsumg}
(it may happen that some $N_{\delta_1,\ldots,\delta_k}(\bW)$ is strictly positive in Corollary \ref{C:lowerdiscrete}). A simple example is provided in Example \ref{E:disbigger}.
On the other hand, if $\calW_1,\ldots,\calW_n$ are sets for which there exists a convex mixed pure subdivision $\calS$ with maximal set of vertices $\calW_1+\cdots+\calW_n$,
then we easily get $D(\calW_1,\ldots,\calW_{n})=A$ so that $D(\calW_1,\ldots,\calW_{n}) \leq MV(P_1,\ldots,P_n)$.

\begin{figure}[ht]
\begin{center}
\includegraphics[scale=0.5]{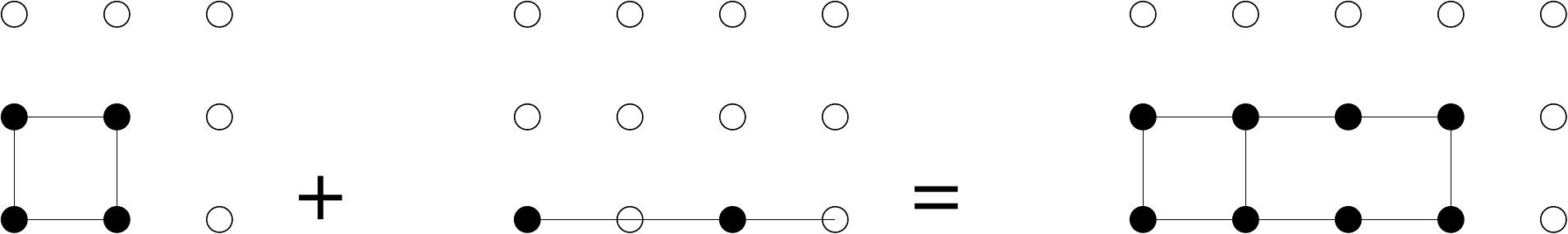}
\caption{}  
  \label{F:DiscExample5}\end{center}
\end{figure}

\begin{example}\label{E:disbigger}
In Figure \ref{F:DiscExample5} we have $\calW_1=\{(0,0),(1,0),(1,1),(0,1)\}$,  $\calW_2=\{(0,0),(2,0)\}$. Then $D(\calW_1, \calW_2)=8-4-2+1 =3 > 2=\mbox{MV}(P_1,P_2)$. If we consider
the quantity $N_{\delta}(\bW)$ associated with $\delta=(0,1)$ and the family $\bW$ formed by the sets $\calW_1,\calW_2$ and $\calW_1+\calW_2$, we retrieve the example given in Figure \ref{F:DiscExample6} providing an excessive point.
\end{example}


There is an important difference between the  discrete mixed volume and the classical one. Namely, as it is well-known the classical mixed volume is increasing, which means that
$MV(P_1,\ldots,P_n) \leq MV(P'_1,\ldots,P'_n)$ if $P_i \subset P'_i$ for $i=1,\ldots,n$.
This is not the case for the discrete mixed volume as shown in Example \ref{E:excplusdefecbis}.

%
%
%
%
%

\begin{ex}\label{E:excplusdefecbis}
Let $\calW_1=\{(0,0),(2,1),(1,2), (2,2)\}$ and $\calW_2=\{(0,0),(2,1),(1,2),(1,1)\}$. These are the sets considered in Example \ref{E:excplusdefec}.
Set $\calW_1'=\calW_1 \cup \{(1,1)\}$ and $\calW_2'=\calW_2$.
Then, $D(\calW_1',\calW_2')=D(\calW_1,\calW_2)-1 < D(\calW_1,\calW_2)$ though $\calW_1 \subset \calW_1'$ and
$\calW_2 \subset \calW_2'$. Adding the point $(1,1)$ to $\calW_1$ removes an excessive point for the mixed subdivision considered in Example \ref{E:excplusdefec}, and decreases the discrete mixed volume by one.
\end{ex}

Recall that $D(\calW_1,\ldots,\calW_{r})=\prod_{i \in {\bf [}r {\bf ]}} (|\calW_i|-1)$ when $\calW_1,\ldots,\calW_r \subset \R^n$ are in general position
(see Proposition \ref{P:bijective}).

\begin{thm}\label{T:boundsforNKouch} 
For any finite sets $\calW_1,\ldots,\calW_{r}$ in  $\R^n$ we have
$$
D(\calW_1,\ldots,\calW_{r}) \leq  \prod_{i \in {\bf [}r {\bf ]}} (|\calW_i|-1).$$
\end{thm}
\begin{proof}
Let $\calS$ be any pure convex mixed subdivision of $P=P_1+\cdots+P_r$ induced by functions $\mu_i: \calW_i \rightarrow \R$.
For $i=1,\ldots,r$, translate each point $w_i \in \calW_i$ into a point $\tilde{w}_i$ by a small vector in order to get a bijective map
between $\calW_i$ and the set $\tilde{\calW}_i$ of points $\tilde{w}_i$, and also a bijection
$\prod_{i=1}^r \tilde{\calW}_i \rightarrow \sum_{i=1}^r \tilde{\calW}_i$, $(\tilde{w}_i)_{i \in [r]} \mapsto \sum_{i=1}^r \tilde{w}_i$.
Then by Proposition \ref{P:bijective}
$$
N(\tilde{\calW}_1,\ldots,\tilde{\calW}_r)=\prod_{i=1}^r (|\tilde{\calW}_i|-1)=\prod_{i=1}^r (|\calW_i|-1).$$
Denote by $\tilde{P}_i$ the convex hull of $\tilde{\calW}_i$. Consider the convex mixed subdivision $\tilde{\calS}$ of $\tilde{P}=\tilde{P}_1+\cdots+\tilde{P}_r$
associated with the functions $\tilde{\mu_i}: \tilde{\calW}_i \rightarrow \R$ defined by $\tilde{\mu_i}(\tilde{w}_i)=\mu_i(w_i)$, where $\tilde{w}_i$ is the image of $w_i$ by
the above bijective map $\calW \rightarrow \tilde{\calW}_i$. If the above translation vectors are small enough, then $\tilde{\calS}$ is again a pure mixed subdivision.
Moreover, the previous bijective maps $\calW_i \rightarrow \tilde{\calW}_i$ induce a bijection $\rho:\calS \rightarrow \tilde{\calS}$ which sends a mixed polytope to a mixed polytope.
Applying Proposition \ref{P:Nsum} simultaneously to $\calS$ and $\tilde{\calS}$ yields
$$N(\calW_1,\ldots,\calW_r)=A+|\mbox{Exc}_{\calS,\delta}|+N_{\delta}(\calW_1,\ldots,\calW_r)$$
and
$$
N(\tilde{\calW}_1,\ldots,\tilde{\calW}_r)=\tilde{A}+|\mbox{Exc}_{\tilde{\calS},\delta}|+N_{\delta}(\tilde{\calW}_1,\ldots,\tilde{\calW}_r),
$$
where $\tilde{A}$ is the sum of the quantities $|(\tilde{\calW}_1+\cdots+\tilde{\calW}_n) \cap (\delta+\tilde{Q})|$ over all mixed polytopes $\tilde{Q}$ of $\tilde{\calS}$.
For a mixed polytope $Q \in \calS$ and its image $\tilde{Q}=\rho(Q) \in \tilde{\calS}$, we have
$$|(\calW_1+\cdots+\calW_r) \cap (\delta+Q)| \leq |(\tilde{\calW_1}+\cdots+\tilde{\calW_r}) \cap (\delta+\tilde{Q})|.$$
This inequality can be strict when a given point of $(\calW_1+\cdots+\calW_r) \cap Q$ can be written in several ways as a sum of points of $\calW_1,\ldots,\calW_r$.
Therefore, $A \leq \tilde{A}$. On the other hand, we obviously have
$|\mbox{Exc}_{\calS,\delta}|=|\mbox{Exc}_{\tilde{\calS},\delta}|$. Thus we are reduced to show that $N_{\delta}(\calW_1,\ldots,\calW_r) \leq N_{\delta}(\tilde{\calW_1},\ldots,\tilde{\calW_r})$.
Applying Proposition \ref{P:Nsumg}, we obtain just as before $A'_j \leq \tilde{A}'_j$ and $|\mbox{Exc}_{\calS'_j,\delta'}|=|\mbox{Exc}_{\tilde{\calS_j},\delta'}|$ for $j=1,\ldots,s$.
Thus, it remains to prove that for small generic linearly independent vectors $\delta_1,\ldots,\delta_n$, we have
$N_{\delta_1,\ldots,\delta_n}(\calW_1,\ldots,\calW_r) \leq N_{\delta_1,\ldots,\delta_n}(\tilde{\calW}_1,\ldots,\tilde{\calW}_r)$. But both members of this inequality vanish due to Proposition \ref{P:zero}.
\end{proof}

The discrete mixed volume is equal to the Kouchnirenko bound when the considered sets are in general position. On the opposite side,
it could be interesting to estimate the discrete mixed volume when all considered sets are equal. When $r=2$ this is easy.


\begin{prop}\label{P:W2}
We have $D(\calW,\calW) \leq \frac{(|\calW|-1)(|\calW|-2)}{2}$.
Thus, a system of two tropical polynomial equations with same support $\calW \subset \R^2$ has at most $\frac{(|\calW|-1)(|\calW|-2)}{2}$ nondegenerate solutions.
\end{prop}
\begin{proof}
Let $s$ be the number of elements of $\calW$. Writing $\calW=\{w_1,\ldots,w_s\}$, we get
$2\calW=\{w_i+w_j \; , \; 1 \leq i < j \leq s\} \cup \{2w_i \; , \; 1 \leq  i \leq s\}$ and thus
$|2\calW| \leq \frac{s(s-1)}{2}+s$. This gives $N(\calW,\calW) \leq \frac{s(s-1)}{2}+s -s-s+1=
\frac{s^2-3s+2}{2}$. \end{proof}

When $|\calW|=4$ the bound in Proposition \ref{P:W2} is $3$ and is sharp, see~\cite{PR}. When $|\calW|=5$, the bound in Proposition \ref{P:W2} is $6$. To our knowledge, it is not known if this bound if sharp or not.


\begin{thebibliography}{1}

\bibitem{BNRSS}
M. Beck, B. Nill, B. Reznick, C. Savage, I. Soprunov, and Z. Xu
{\it Let me tell you my favorite lattice-point problem . . . . }
In Integer points in polyhedra-geometry, number theory, representation theory, algebra, optimization, statistics, volume 452 of Contemp. Math., pages 179-187. 
Amer. Math. Soc., Providence, RI, 2008.

\bibitem{BS} M. Beck and F.~Sottile
{\it Irrational proofs for three theorems of Stanley.} European J. Combin. 28 (2007), no. 1, 403-409.

\bibitem{B}
D. N. Bernstein,
{\it The number of roots of a system of equations}, Funct. Anal. Appl. 9 (1975),
183–185.


\bibitem{BB}
B. Bertrand and F. Bihan {\it Intersection multiplicity numbers between tropical hypersurfaces}.
Algebraic and combinatorial aspects of tropical geometry, 1--19, Contemp. Math., 589, Amer. Math. Soc., Providence, RI, 2013.

\bibitem{Bi}  F. Bihan {\it Viro method for the construction of real complete intersections,}
Advances in Mathematics 169 no 2, (2002), 177--186.

\bibitem{BS07}
F.~Bihan and F.~Sottile, {{\it   New fewnomial upper bounds from {G}ale dual
  polynomial systems}}, 2007, Moscow Mathematical Journal, Volume 7, Number 3, 387--407. 

\bibitem{BJSST07}T. Bogart, A.N. Jensen, D. Speyer, B. Sturmfels and R.R. Thomas
{\it  Computing tropical varieties},
J. Symb. Comput., vol. 422, issues 1-2, (2007), 54--73.


\bibitem{CLO} D. Cox, J. Little and D. O'shea,
{\it  Using Algebraic Geometry},
Second edition. Graduate Texts in Mathematics, 185. Springer, New York, (2005). 

\bibitem{Gathmann} A. Gathmann, {\it  Tropical algebraic geometry},
Jahresber. Deutsch. Math.-Verein.  108,  no. 1, 3--32 (2006).

\bibitem{Ha}B. Haas, {\it  A Simple Counter-Example to Kushnirenko's Conjecture,}
Beitrage Zur Algebra und Geometrie, Vol. 43, no. 1, 1-- 8 (2002).

\bibitem{IMS}
I. Itenberg, G. Mikhalkin and E. Shustin,
{\it  Lecture notes of the Oberwolfach seminar "Tropical Algebraic Geometry"}, 
Birkhauser, Oberwolfach Seminars Series, Vol. 35, (2007).

\bibitem{IR} I. Itenberg and M.-F. Roy  {\it  Multivariate Descartes' Rule},
Beitrage Zur Algebra und Geometrie, Vol. 37, no. 2, 337--346 (1996).

\bibitem{Kho} A. Khovanskii, {\it Fewnomials}, Trans. of Math. Monographs, 88, AMS, 1991.

\bibitem{Kho1} A. Khovanskii, {\it Newton polyhedra and the genus of complete intersections}
Funktsional. Anal. i Prilozhen., 12(1) : 51--61, 1978.

\bibitem{LRW03}
Tien-Yien Li, J.~Maurice Rojas, and Xiaoshen Wang,{\it  Counting real
  connected components of trinomial curve intersections and {$m$}-nomial
  hypersurfaces}, Discrete Comput. Geom. \textbf{30} (2003), no.~3, 379--414.
  \MR{2 002 964}


\bibitem{LW} T. Li and X. Wang
{\it  On multivariate Descartes' Rule - A counter-example}
Beiträge zur Algebra und Geometrie {\bf 39} (1998), no. 1, 1--5.

\bibitem{Mtropappli} G. Mikhalkin, {\it  Tropical Geometry and its applications},
Proceedings of the International Congress of Mathematicians, Madrid 2006, 827--852.

\bibitem{N} B. Nill, {The mixed degree of families of lattice polytopes,} in
preparation.
 
\bibitem{PR} K. Phillipson and J.~Maurice Rojas {\it
Fewnomial Systems with Many Roots, and an Adelic Tau Conjecture,}
proceedings of Bellairs workshop on tropical and non-Archimedean geometry (May 6-13, 2011, Barbados), Contemporary Mathematics, vol. 605, 45--71, AMS Press, to appear.

\bibitem{RGST} J. Richter-Gebert, B. Sturmfels and  T. Theobald,
{\it  First steps in tropical geometry},
Idempotent mathematics and mathematical physics,  289--317, 
Contemp. Math., 377, Amer. Math. Soc., Providence, RI, (2005). 

\bibitem{ST}
R. Steffens and  T. Theobald
{\it Combinatorics and genus of tropical intersections and Ehrhart theory.}
SIAM J. Discrete Math. 24 (2010), no. 1, 17--32.


\bibitem{St} B. Sturmfels, 
{\it  Viro's theorem for complete intersections},
Ann. Scuola Norm. Sup. Pisa Cl. Sci. (4), vol. 21, no. 3, 377--386, (1994).


\end{thebibliography}
\end{document}